\documentclass[onefignum,onetabnum,dvipsnames]{siamart190516}
\usepackage{amsfonts,amsmath,amssymb,bm,hyperref,algorithm,algpseudocode,algorithmicx,setspace,arydshln}

\title{Krylov subspace residual and restarting for certain second order
differential equations}

\author{%
M.A. Botchev%
\thanks{Keldysh Institute of Applied Mathematics,
Russian Academy of Sciences, Miusskaya~Sq.~4, Moscow 125047,
Russia,
\email{botchev@ya.ru}.
}
\and
L.A. Knizhnerman%
\thanks{Marchuk Institute of Numerical Mathematics, Russian Academy of Sciences,
Gubkin St.~8, Moscow 119333, Russia, \email{lknizhnerman@gmail.com}.
Work of this author is supported by the Moscow Center of Fundamental and Applied Mathematics at INM RAS (Agreement No.~075-15-2022-286 with the Ministry of Education and Science of the Russian Federation).}
\and
M. Schweitzer%
\thanks{School of Mathematics and Natural Sciences, Bergische Universit\"at Wuppertal, 42097 Wuppertal, Germany, \email{marcel@uni-wuppertal.de}.}
}

\newcommand{\Cc}{\mathbb{C}}
\newcommand{\calO}{\mathcal{O}}

\newcommand{\geqs}{\geqslant}

\newcommand{\leqs}{\leqslant}

\renewcommand{\Re}{\mathrm{Re}}
\newcommand{\Rr}{\mathbb{R}}
\newcommand{\Rrnn}{\mathbb{R}^{n\times n}}
\newcommand{\tol}{\mathtt{tol}}
\newcommand{\vx}{\bm{x}}

\def\be#1\ee{\begin{equation}#1\end{equation}}
\newcommand{\bea}{\begin{eqnarray}}
\newcommand{\eea}{\end{eqnarray}}
\newcommand{\beas}{\begin{eqnarray*}}
\newcommand{\eeas}{\end{eqnarray*}}

\newcommand{\tT}{{\mbox{\it \tiny T}}}

\renewcommand{\Re}{\mathop{\rm Re}\nolimits}

\DeclareMathOperator{\Real}{Re}

\DeclareMathOperator{\nnz}{nnz}

\newtheorem{remark}[theorem]{Remark}
\let\oldremark\remark
\renewcommand{\remark}{\oldremark\normalfont}

\usepackage{pgfplots}
\usepgfplotslibrary{groupplots}
\usetikzlibrary{external}

\pgfkeys{/pgf/images/include external/.code={\includegraphics{#1}}}

\usetikzlibrary{decorations.markings}
\makeatletter
\tikzset{
  nomorepostactions/.code={\let\tikz@postactions=\pgfutil@empty},
  mymark/.style 2 args={decoration={markings,
    mark= between positions 0 and 1 step (1/9)*\pgfdecoratedpathlength with{%
        \tikzset{#2,every mark}\tikz@options
        \pgfuseplotmark{#1}%
      },  
    },
    postaction={decorate},
    /pgfplots/legend image post style={
        mark=#1,mark options={#2},every path/.append style={nomorepostactions}
    },
  },
}
\makeatother

\pgfplotscreateplotcyclelist{list_std}{%
Cerulean,line width=1.25pt, mark=x, solid\\
BurntOrange,line width=1.25pt, mark=o, solid\\
OliveGreen,line width=1.25pt, mark=diamond, solid\\
Orchid,line width=1.25pt, mark=square, solid\\
}
\pgfplotsset{compat=1.17}
\begin{document}
\maketitle

\begin{abstract}
We propose algorithms for efficient time integration of large systems of oscillatory second order ordinary differential equations (ODEs) whose solution can be expressed in terms of trigonometric matrix functions. Our algorithms are based on a residual notion for second order ODEs, which allows to extend the ``residual-time restarting'' Krylov subspace framework---which was recently introduced for exponential and $\varphi$-functions occurring in time integration of first order ODEs---to our setting. We then show that the computational cost can be further reduced in many cases by using 
our restarting %MB residual notion for step size selection 
in the Gautschi cosine scheme. We analyze residual convergence in terms of Faber and Chebyshev series and supplement these theoretical results by numerical experiments illustrating the efficiency of the proposed methods.
\end{abstract}



\section{Introduction}
Exponential time integration has been a rapidly developing research
area, and exponential time solvers have been shown to be efficient in
various application classes.  These time integration schemes typically
involve a matrix function (e.g., a matrix exponential or cosine) and
possess a characteristic property of being exact for a certain 
simplified problem~\cite{HochbruckOstermann2010}.
Although development of exponential schemes started about 60~years
ago~\cite{Certaine1960,Lawson1967,Legras1966,Norsett1969,vdHouwenVerwer1974,Verwer1977},
their successful application for large scale problems was delayed for about
30~years, until numerical
linear algebra tools to evaluate the matrix functions, in particular Krylov subspace methods, 
matured (we list chronologically some of the first relevant
papers~\cite{ParkLight86,Henk:f(A),DruskinKnizhnerman1989,Knizh91,Saad1992,GallopoulosSaad1992,DruskinKnizhnerman1994,DruskinKnizhnerman1995,HochbruckLubich1997}).

Although Krylov subspace iterations can be quite successful within
exponential time integration methods, their efficient implementation
is often a challenging task, especially for large scale and stiff problems.
This is because %, for matrix functions such as the matrix exponential and cosine,
conventional Krylov subspace methods, especially those based on the
Arnoldi process, become more work and memory consuming 
as the iteration number grows, which can lead to an efficiency degradation.
To avoid this and to keep the Krylov subspace dimension
restricted, a number of approaches has been developed.
First, to accelerate convergence in Krylov subspace methods,
rational shift-and-invert Krylov subspace methods can be
employed~\cite{MoretNovati2004,vdEH06}.
These methods typically exhibit a rapid and mesh-independent convergence
for parabolic problems~\cite{Grimm2012,MoretNovati2004,vdEH06} at a price of solving a linear system at
each Krylov iteration.
Another approach to keep the Krylov subspace dimension restricted is
to adjust the time step size.  In the EXPOKIT package this is done
with a sophisticated error estimation procedure~\cite{EXPOKIT}.
Furthermore, to control the Krylov subspace dimension, various restarting
approaches are designed.  These, among others, include restarting for general
matrix functions~\cite{AfanasjewEtAl2008a,EiermannErnst2006,EiermannErnstGuettel2011,FrommerGuettelSchweitzer2014a,Niehoff2006,TalEzer2007} and restarting designed specifically for the matrix exponential and related functions~\cite{BotchevGrimmHochbruck2013,CelledoniMoret1997,DruskinGreenbaumKnizhnerman1998}.

An efficient restarting strategy specifically intended for the matrix exponential
and the related $\varphi$-function was recently proposed in~\cite{BotchevKnizhnerman2020,BotchevKnizhnermanTyrtyshnikov2020}.
These are the matrix functions which are essential in the exponential
integration of first order ODE systems.
However, as far as we know, no restarting strategies have been designed so far
specifically for the cosine, sine and other related matrix functions, which are
instrumental in exponential time integration of second order ODE systems.
This is remarkable because for the matrix functions arising in numerical treatment
of second order ODEs, Krylov subspace methods usually converge much slower than for
the matrix exponential, see, e.g.,~\cite[Theorems~4, 6, 7]{DruskinKnizhnerman1989} and
our observations below.
Hence, efficient restarting techniques for these functions
would be a very welcome contribution.
More precisely, for $A=A^T\in\Rrnn$ and $t>0$,
convergence of Krylov subspace methods to compute actions of the matrix functions
$\cos(t\sqrt{A})$ and $(tA)^{-1/2}\sin(t\sqrt{A})$ is similar
to convergence of Krylov subspace methods to compute actions of
\begin{equation}
\label{skew_exp}
\exp (-t \mathcal{A}), \qquad
\mathcal{A} = 
\begin{bmatrix}
0 & -I \\ A & 0  
\end{bmatrix},
\end{equation}
where $I\in\Rrnn$ is the identity matrix. 
Since the $2\times 2$ block matrix $\mathcal{A}$ in~\eqref{skew_exp} is similar to
a skew-symmetric matrix, Krylov subspace methods for $\exp(-t \mathcal{A})$ converge 
much slower than for the matrix exponential $\exp(-tA)$ with symmetric $A$,
which occur in time integration of first order ODE systems. % $\bm{y}'(t)=-A\bm{y}(t)$.
This significant difference in Krylov subspace method convergence for symmetric
versus skew-symmetric matrices is well-known, see, e.g.,~\cite[Theorems~2, 4]{HochbruckLubich1997}. 

The goal of this paper is to fill this gap by introducing an efficient restarting
technique for matrix functions occurring in second order ODEs.
The restarting approach we follow is the residual-time (RT) restarting introduced
in~\cite{BotchevKnizhnerman2020,BotchevKnizhnermanTyrtyshnikov2020}.  Although this RT~restarting appears to
work well for these matrix functions,
we show how the computational costs can further be reduced by
combining our restarting with the Gautschi cosine scheme~\cite{HochbruckLubich1999}.
More specifically,
the time dependence of the residual is employed to find a proper
time step size for the Gautschi scheme.

The rest of the paper is organized as follows. In Section~\ref{sec:residual}, we briefly discuss Krylov subspace methods for approximating the trigonometric matrix functions occurring in the time integration of second order ODEs and introduce our residual notion. Based on this residual notion, a residual-time restarting algorithm as well as a Gautschi cosine scheme with residual-based step size selection are introduced. A detailed analysis of the residual convergence in terms of Faber series (for general $A$) and Chebyshev series (for symmetric $A$) is given in Section~\ref{sec:residual_convergence}. Results of numerical experiments illustrating the performance of the proposed methods are reported in Section~\ref{sec:experiments}, followed by some concluding remarks in Section~\ref{sec:conclusions}.

Throughout this paper $\|\cdot\|$ denotes the Euclidean vector norm $\|\cdot\|_2$ or the corresponding
matrix operator norm.

\section{Residual and restarting in Krylov subspace methods for second
order ODEs}\label{sec:residual}
\subsection{Problem setting}
Let $A \in \Rrnn$ be a matrix such that
\begin{equation*}
\label{pos_real}
\Real(\vx^H\!A\vx) \geqs 0 \quad \text{for all }\vx \in \Cc^n.  
\end{equation*} 
We are interested in solving an IVP
\begin{equation}
\label{ivp}
\bm{y}''(t) = -A \bm{y}(t) + \bm{g},\quad
\bm{y}(0)=\bm{u},\quad \bm{y}'(0)=\bm{v}.
\end{equation}
Initial value problems of the form~\eqref{ivp} arise, e.g.,
from semi-discretized wave equations within the method of lines.

Let $\psi$ and $\sigma$ be entire functions defined as
\begin{equation}
\label{psi_sigma}
\psi(x^2)=2\frac{1-\cos x}{x^2}, \quad
\sigma(x^2) = \frac{\sin x}{x},
\end{equation}
where we set, by definition, $\psi(0)=1$ and $\sigma(0)=1$. It can easily be checked (cf.~Appendix~\ref{sec:appendix_new}) that the exact solution $\bm{y}(t)$ of~\eqref{ivp} 
reads, for any $t\geqs 0$,
\begin{equation}
\label{yex}
\bm{y}(t) = \bm{u} + \frac12 t^2\psi(t^2 A)(-A\bm{u}+\bm{g})
                    + t\sigma(t^2A)\bm{v}.   
\end{equation}

Note that, due to~\eqref{psi_sigma}, the matrix functions 
$$
\psi(t^2 A) = 2(t^2 A)^{-1}(I - \cos(t\sqrt{A})),\quad
\sigma(t^2 A) = (t\sqrt{A})^{-1}\sin(t\sqrt{A})
$$
are defined also for singular $A$. Furthermore, $\sqrt{A}$ can be taken here to be any
(not necessarily primary) branch of the square root of $A$; see, e.g.,~\cite[Chapter~2, Section~1]{Higham2008}.

\subsection{Krylov subspace approximations and their residuals}\label{subsec:krylov}
To use~\eqref{yex} in practical computations, two matrix functions
have to be evaluated.  In this work this is done by Krylov subspace methods
with the Arnoldi or Lanczos process in the following usual way.

Let $f(A)=\frac{t^2}2\psi(t^2 A)$, where we consider $t\geqs 0$ as a parameter, and let
$\bm{w}=-A\bm{u}+\bm{g}$.  
We compute $f(A)\bm{w}$ approximately by using the Krylov subspace
$$
\mathcal{K}_m(A,\bm{w})=
\mathrm{span} \left\{ \bm{w}, A\bm{w}, A^2\bm{w},\dots, A^{m-1}\bm{w}\right\}.
$$
We then take $\bm{v}_1=\bm{w}/\beta$, with $\beta=\|\bm{w}\|$, as the first Krylov subspace
basis vector and carry out $m$ steps of the Arnoldi
process~\cite{Saad2003,VanDerVorst2003}, building a matrix
$V_{m+1}\in\Rr^{n\times (m+1)}$ with orthonormal columns $\bm{v}_1$, \dots, $\bm{v}_{m+1}$
and an upper-Hessenberg matrix $\underline{H}_m=V_{m+1}^HAV_m\in\Rr^{(m+1)\times m}$.
Then $\mathrm{colspan} (V_m)=\mathcal{K}_m(A,\bm{w})$ and
the so-called Arnoldi decomposition holds:
\begin{equation}
\label{eq:arnoldi_decomposition}
\begin{aligned}
AV_m &= V_{m+1}\underline{H}_m,
\\
\text{or}\quad
AV_m &= V_mH_m + h_{m+1,m} \bm{v}_{m+1} \bm{e}_m^T,    
\end{aligned}
\end{equation}
where $H_m=V_m^HAV_m\in\Rr^{m\times m}$ is composed of the first $m$ rows of $\underline{H}_m$
and $\bm{e}_m=(0,\dots,0,1)^T\in \Rr^m$ denotes the $m$th canonical unit vector.
If the matrix $A$ is symmetric, instead of the Arnoldi process, 
the Lanczos process is employed and relation~\eqref{eq:arnoldi_decomposition} holds 
for a tridiagonal matrix $\underline{H}_m$.
The main computational effort for the Arnoldi process lies in $m$ matrix vector products with 
$A$---which have an overall complexity $\calO(\nnz(A)\cdot m)$ for sparse $A$ 
with $\nnz(A)$ nonzero entries---and in the orthogonalization which requires $\calO(m^2n)$ operations 
in total.

In view of~\eqref{eq:arnoldi_decomposition} and the equality $\bm{w}=V_m (\beta \bm{e}_1)$,
an approximation $\bm{y}_m$ to $\bm{y}=f(A)\bm{w}$ can be found by projecting 
onto the Krylov subspace via
\begin{equation}\label{eq:arnoldi_approximation}
\bm{y} = f(A)\bm{w} \approx   
\bm{y}_m = V_m f(V_m^HAV_m)V_m^H\bm{w}=
V_m f(H_m)(\beta \bm{e}_1).
\end{equation}
This approximation can also be derived by noticing that $\frac{t^2}2\psi(t^2A)\bm{w}$ 
is the solution of the IVP
\begin{equation*}
\label{ivp_psi}
\bm{y}''(t) = -A \bm{y}(t) + \bm{g}-A\bm{u},\quad
\bm{y}(0)=0,\quad \bm{y}'(0)=0. 
\end{equation*}
Galerkin projection of this IVP onto the same Krylov subspace then leads to
the approximation $\bm{y}_m(t)\approx\bm{y}(t)$, with 
\begin{equation}
\label{Gal_proj_psi}
\bm{y}_m(t) = V_m u(t), \qquad
\left\{\begin{aligned}
& u''(t) = -H_m u(t) + \beta \bm{e}_1,\\
& u(0)=0,\quad u'(0)=0.   
\end{aligned}\right.
\end{equation}
It is not difficult to check that relations~\eqref{eq:arnoldi_approximation}
and~\eqref{Gal_proj_psi} yield the same approximation~$\bm{y}_m(t)$
(indeed, recall that $f$ depends on the parameter $t\geqs 0$, so that
$u(t)=f(H_m)(\beta \bm{e}_1)$).

The same approach, where we set $f(A)=t\sigma(t^2A)$ and $\bm{w}=\bm{v}$,
can also be used to evaluate
the term $t\sigma(t^2A)\bm{v}$ in~\eqref{yex}.  In this case we have
\begin{equation}
\label{Gal_proj_sigma}
\bm{y}_m(t) = V_m u(t), \qquad
\left\{\begin{aligned}
& u''(t) = -H_m u(t),\\
& u(0)=0,\quad u'(0)=\beta \bm{e}_1.   
\end{aligned}\right.
\end{equation}
Note that the matrices $V_m$ and $H_m$ here differ from their counterparts in~\eqref{Gal_proj_psi}.

Evaluating \eqref{eq:arnoldi_approximation} requires computing $f(H_m)$, which 
has a computational cost of $\calO(m^3)$ for most matrix functions. Additionally, even when $A$ is symmetric, 
the complete Krylov basis $V_m$ is required for forming $\bm{y}_m$. 
For very large values of $n$, as they frequently appear in applications, the storage 
requirements for $V_m$ can therefore limit the number of iterations that can be performed. % before storing further basis vectors becomes impossible. 
Additionally, when a large number $m$ of iterations is necessary for reaching the desired 
accuracy, the cost for computing $f(H_m)$ might also become significant. 

As discussed above, several different methods for \emph{restarting} the Arnoldi iteration have been 
proposed to tackle this problem.  All of these methods have in common that they allow to discard the 
Arnoldi basis vectors %$V_{m_{\max}}$
after a small, fixed number $m_{\max}$ of steps and then start a new \emph{cycle} 
of the method in which $m_{\max}$ further Arnoldi iterations are performed. For further details on restarted Krylov methods for matrix functions, we refer the reader 
to~\cite{BotchevKnizhnerman2020,BotchevKnizhnermanTyrtyshnikov2020,EiermannErnst2006,FrommerGuettelSchweitzer2014b, FrommerGuettelSchweitzer2014a,Schweitzer2016,TalEzer2007} and the references therein.

To control the accuracy of the Krylov subspace iterative approximations $\bm{y}_m$,
we introduce residuals for the matrix functions $\psi(t^2A)$ and $\sigma(t^2A)$.
To do so, we follow the approach
of~\cite{BotchevGrimmHochbruck2013,CelledoniMoret1997,DruskinGreenbaumKnizhnerman1998,EiermannErnst2006}
and view a matrix vector product with each of these matrix functions as
an exact solution of a certain IVP.  Hence, the corresponding residual
can naturally be defined with respect to this IVP.
These residuals for the $\psi$ and $\sigma$ matrix functions and, for comparison
purposes, also for the matrix exponential, are given in Table~\ref{t:res}.
It is easy to check that the residuals
of all the Krylov subspace approximations given in the table
have the form
\begin{equation}
\label{rm}
\bm{r}_m(t)=-\beta_m(t)\bm{v}_{m+1},  \quad \beta_m(t)=h_{m+1,m}\bm{e}_m^Tu(t),
\end{equation}
where $u(t)$ solves the corresponding projected IVP.
For instance, for the residual $\bm{r}_m(t)$ of the $\psi(t^2A)$ function, we
obtain, using~\eqref{eq:arnoldi_decomposition} and relation~$\bm{g}-A\bm{u}=\beta\bm{v}_1$,
\begin{align*}
  \bm{r}_m(t) &\equiv -A\bm{y}_m(t)+ \bm{g}-A\bm{u}-\bm{y}_m''(t)  =
  -AV_mu(t)+ \bm{g}-A\bm{u}-V_mu''(t) \\
  &= -(V_mH_m+ h_{m+1,m} \bm{v}_{m+1} \bm{e}_m^T)u(t)+ \bm{g}-A\bm{u}-V_m(-H_mu(t) + \beta \bm{e}_1)  \\
  &= -h_{m+1,m} \bm{v}_{m+1} \bm{e}_m^T u(t),
\end{align*}
which yields~\eqref{rm}.
Thus, for all the three matrix functions in Table~\ref{t:res} the residual is a scalar
function $\beta_m(t)$ of time $t$ multiplied with the last Krylov subspace  basis vector $\bm{v}_{m+1}$.
Hence, since the columns of $V_{m+1}$ are orthonormal,
the residual satisfies the Galerkin property
$$
V_m^H \bm{r}_m(t)=0.
$$

\begin{table}
\caption{Residuals for the matrix functions $\psi$ and $\sigma$ compared to
that of the matrix exponential.  
In case the Krylov subspace method~\eqref{eq:arnoldi_approximation}
is used, each of the residuals satisfies
relation~\eqref{rm} with $u(t)$ being the solution of the projected
IVP, as given in the table.}
\label{t:res}
\begin{center}
\renewcommand{\arraystretch}{1.2}
\scalebox{.96}{\begin{tabular}{ccc}
\hline\hline
$f(A)$ action as an   & IVP solved      & residual for $\bm{y}_m(t) = V_mu(t) \approx\bm{y}(t)$ \\
exact solution        & with $f(A)$     &  and projected IVP   \\ 
\hline
$\bm{y}(t)=\exp(-tA)\bm{v}$  
               & $\left\{\begin{aligned}
                  & \bm{y}'(t)=-A\bm{y}(t)\\
                  & \bm{y}(0)=\bm{v}
                  \end{aligned}\right.$ 
                              & $\begin{gathered}
                                 \bm{r}_m(t)\equiv -A\bm{y}_m(t)-\bm{y}_m'(t) \\[1ex] 
                                 \left\{\begin{aligned}
                                 u'(t) &=-H_mu(t) \\
                                 u(0)  &=\beta \bm{e}_1  
                                 \end{aligned}\right.
                                 \end{gathered}$ 
\\\hline
$\bm{y}(t)=t\sigma(-t^2A)\bm{v}$  
               & $\left\{\begin{aligned}
                  & \bm{y}''(t)=-A\bm{y}(t)\\
                  & \bm{y}(0)=0,\; \bm{y}'(0)=\bm{v}
                  \end{aligned}\right.$
                              & $\begin{gathered}
                                  \bm{r}_m^{(\sigma)}(t)\equiv -A\bm{y}_m(t)-\bm{y}_m''(t)\\[1ex]
                                  \left\{\begin{aligned}
                                  & u''(t)=-H_mu(t)\\
                                  & u(0)=0,\; u'(0)=\beta \bm{e}_1  
                                  \end{aligned}\right.
                                  \end{gathered}$
\\\hline
$
\begin{gathered}
\tilde{\bm{y}}(t) = \frac{t^2}2\psi(-t^2A)\tilde{\bm{g}} %MB (\bm{g}-A\bm{u})  
\end{gathered}$
               & $\left\{\begin{aligned}
                  & \tilde{\bm{y}}''(t)=-A\tilde{\bm{y}}(t)+\tilde{\bm{g}}\\ %MB -A\bm{u} \\
                  & \tilde{\bm{y}}(0)=0,\; \tilde{\bm{y}}'(0) = 0
                 \end{aligned}\right.$
                                 & $\begin{gathered}
                                     \tilde{\bm{r}}_m^{(\psi)}(t) \equiv - A\tilde{\bm{y}}_m(t)+\tilde{\bm{g}} %MB -A\bm{u}
                                     -\tilde{\bm{y}}_m''(t)
                                 \\[1ex]
                                  \left\{\begin{aligned}
                                  & u''(t)=-H_mu(t) + \beta \bm{e}_1\\
                                  & u(0)=0,\; u'(0)=0
                                  \end{aligned}\right.
                                  \end{gathered}$
\\\hline
\multicolumn{3}{l}{NB: IVP 
$\left\{\begin{aligned}
                   & \bm{y}''(t)=-A\bm{y}(t)+\bm{g}\\
                   & \bm{y}(0)=\bm{u},\; \bm{y}'(0) = 0
                  \end{aligned}\right.\quad$ can be solved as $\bm{y}(t)=\tilde{\bm{y}}(t)+\bm{u}$ with
                  $\tilde{\bm{y}}(t)$ defined}
\\
\multicolumn{3}{l}{above for $\tilde{\bm{g}}:=\bm{g}-A\bm{u}$.
We then have $\bm{r}_m^{(\psi)}(t) \equiv - A\bm{y}_m(t)+\bm{g}-\tilde{\bm{y}}_m''(t)\equiv\tilde{\bm{r}}_m^{(\psi)}(t)$.}
\\\hline
\end{tabular}}%
\end{center}
\addtolength{\tabcolsep}{-0.12em}
\end{table}

Using~\eqref{Gal_proj_psi} and~\eqref{Gal_proj_sigma},
we can compute a Krylov subspace approximation $\bm{y}_m(t)$ to the exact
solution $\bm{y}(t)$, cf.~\eqref{yex}, as a sum
of two approximate matrix function actions
\begin{equation}
\label{ym}
\begin{aligned}
\bm{y}_m(t) = \bm{u} + \bm{y}_m^{(\psi)}(t) + \bm{y}_m^{(\sigma)}(t),
\quad
\bm{y}_m^{(\psi)}(t) &\approx \frac12 t^2\psi(t^2 A)(-A\bm{u}+\bm{g}),
\\
\bm{y}_m^{(\sigma)}(t) &\approx t\sigma(t^2A)\bm{v},  
\end{aligned}
\end{equation}
where $\bm{y}_m^{(\psi)}(t)$ and $\bm{y}_m^{(\sigma)}(t)$ are the Krylov
subspace approximations defined by \eqref{Gal_proj_psi} and \eqref{Gal_proj_sigma},
respectively. For simplicity of presentation, we assume
that both approximations $\bm{y}_m^{(\psi)}(t)$ and $\bm{y}_m^{(\sigma)}(t)$
are computed with the same number of Krylov steps $m$. 
In practice each of them is computed with a certain number of Krylov steps,
such that its residual (defined in Table~\ref{t:res}) is sufficiently
small in norm. For later reference, we note that the time derivative of the 
approximate solution $\bm{y}_m(t)$ is given by 
\begin{equation}
\label{vm}
\begin{aligned}
\bm{y}_m^\prime(t) = \bm{v}_m(t) := \bm{v}_m^{(\psi)}(t) + \bm{v}_m^{(\sigma)}(t),
\quad
\bm{v}_m^{(\psi)}(t) &\approx t\sigma(t^2A)(-A\bm{u}+\bm{g}),
\\
\bm{v}_m^{(\sigma)}(t) &\approx \left(I-\frac12 t^2A\psi(t^2 A)\right)\bm{v},  
\end{aligned}
\end{equation}
where we have used the fact that
\begin{equation}\label{eq:first_derivatives_psi_sigma}
\frac{\text{d}}{\text{d} t}\ \frac12t^2\psi(t^2A) = t\sigma(t^2A)\quad
\text{ and }\quad\frac{\text{d}}{\text{d} t}\  t\sigma(t^2A) = I-\frac12t^2A\psi(t^2 A).
\end{equation}

The residual $\bm{r}_m(t)$ of the approximate solution $\bm{y}_m(t)$
in~\eqref{ym}
can also be split into the parts corresponding to $\bm{y}_m^{(\psi)}(t)$ and
$\bm{y}_m^{(\sigma)}(t)$:
\begin{equation}
\label{rm_split}
\begin{aligned}
\bm{r}_m(t) &\equiv -A\bm{y}_m(t)+ \bm{g}-\bm{y}_m''(t) =
\\
&= -A (\bm{u} + \bm{y}_m^{(\psi)}(t) + \bm{y}_m^{(\sigma)}(t)) +\bm{g}
-(\bm{y}_m^{(\psi)}(t) + \bm{y}_m^{(\sigma)}(t))''
\\
&= -A\bm{y}_m^{(\psi)}(t) +\bm{g} -A\bm{u}-(\bm{y}_m^{(\psi)}(t))'' 
   -A\bm{y}_m^{(\sigma)}(t)-(\bm{y}_m^{(\sigma)}(t))''
\\
&= \bm{r}_m^{(\psi)}(t) + \bm{r}_m^{(\sigma)}(t) ,
\end{aligned}
\end{equation}
where both the residuals $\bm{r}_m^{(\psi)}(t)$ and $\bm{r}_m^{(\sigma)}(t)$
are defined in Table~\ref{t:res} and satisfy relation~\eqref{rm}.

\subsection{Transformation to a first order ODE system}\label{subsec:firstoder}
A possible way to solve IVP~\eqref{ivp} is to introduce
\begin{equation}
\label{mtx2}  
\bm{w}(t) = \begin{bmatrix}\bm{y}(t)\\ \bm{y}'(t) \end{bmatrix},
\quad
\bm{w}_0 = \begin{bmatrix}\bm{u} \\ \bm{v} \end{bmatrix},
\quad
\widehat{\bm{g}} = \begin{bmatrix}0 \\ \bm{g} \end{bmatrix},
\qquad
\mathcal{A} = 
\begin{bmatrix}
0 & -I \\ A & 0  
\end{bmatrix}
\end{equation}
and to consider an equivalent IVP 
\begin{equation}
\label{ivp1}
\bm{w}'(t)=-\mathcal{A}\bm{w}(t) + \widehat{\bm{g}},
\quad 
\bm{w}(0)=\bm{w}_0.  
\end{equation}
This approach has a number of drawbacks, such as the necessity of working with (and storing) vectors of twice the size and the loss of possible symmetry of $A$ in $\mathcal{A}$. Another important drawback of this approach is that problem~\eqref{ivp1} is unstable
in the following sense. For an IVP of the form~\eqref{ivp1}, stability estimates with respect to perturbations in $\bm{w}_0$ and $\bm{g}$ (see, e.g.,~\cite[formula~(I.2.22)]{HundsdorferVerwer2007}) are typically obtained with the help of the constants $\omega,C\in\Rr$ such that
\begin{equation}
\label{stab}  
\|\exp(-t\mathcal{A})\|\leqs C e^{-t\omega}.
\end{equation} 
For the 2-norm (used in this paper) the last estimate
is satisfied with $C=1$ and $\omega$ being the smallest 
eigenvalue of $\mathcal{A}_H=\frac12(\mathcal{A}+\mathcal{A}^T)$.
If $\omega\geqs 0$ then estimate~\eqref{stab} is obviously much more useful
than a straightforward estimate $\|\exp(-t\mathcal{A})\|\leqs e^{t\|\mathcal{A}\|}$,
which formally shows stability, too, but leads to a large over-estimation
of the error propagation. 
However, for the matrix $\mathcal{A}_H$ defined with respect to $\mathcal{A}$ from~\eqref{mtx2}
we have that its smallest eigenvalue is 
$\omega=-\|\mathcal{A}_H\|=-\frac12\|A-I\|$.
Hence, for $A$ large in norm (a typical situation for space-discretized PDEs)
no sensible stability estimates can be obtained in this way.

\subsection{The RT restarting procedure}\label{subsec:rt_restarting}
The residual-time (RT) restarting procedure~\cite{BotchevKnizhnerman2020,BotchevKnizhnermanTyrtyshnikov2020} 
is developed for Krylov subspace methods
applied to first order ODE system, i.e., for ODE systems of the form
$\bm{y}'(t)=-A\bm{y}(t)+\bm{g}$.  In this case the residual for an approximate
solution $\bm{y}_m(t)$ is defined as
$$
\bm{r}_m(t)\equiv -A\bm{y}_m(t)+\bm{g}-\bm{y}_m'(t).
$$
The RT restarting is based on the observation
that, for any fixed $m$, $\|\bm{r}_m(t)\|$ is arbitrarily small for sufficiently small $t$.
It turns out that a similar property holds also for the residuals of the second
order equations appearing in Table~\ref{t:res}. The following result is formulated in terms of the function
$$\varphi(z) = \frac{e^z-1}{z}.$$

\begin{proposition}
\label{ut}
Let $H_m\in\Rr^{m\times m}$ be the matrix obtained by the Arnoldi or Lanczos
process, defined in~\eqref{eq:arnoldi_decomposition}, and let $u(t):\Rr\rightarrow\Rr^m$ be 
the solution of the projected IVP 
\begin{equation*}
\label{ivpH}
u''(t) = -H_m u(t) + \bm{g}, \quad u(0) = 0,\quad u'(0)=\bm{v}_0,   
\end{equation*}
with $\bm{g},\bm{v}_0\in\Rr^m$ given.  Then
\begin{equation}
\label{ut_est}
\|u(t)\|\leqs t\varphi(-t\hat{\omega}_m)\sqrt{\|\bm{v}_0\|^2 + \|\bm{g}\|^2},  
 \quad t\geqs 0,
\end{equation}
where $\hat{\omega}_m=-\frac12\|H_m-I\|$.
Thus, for any $\varepsilon>0$ there exists $\delta>0$ such that
$\|u(t)\|\leqs \varepsilon$ for $t\in[0,\delta]$.
\end{proposition}

\begin{proof}
The last statement (that $\|u(t)\|$ is arbitrarily small for 
sufficiently small $t>0$) follows from the fact that $u(t)$
is infinitely many times differentiable,
see, e.g.,~\cite{Nefedov_eaODE}. 
To prove estimate~\eqref{ut_est}, consider
$$
w(t) = \begin{bmatrix}u(t)\\ u'(t) \end{bmatrix},
\quad
\bm{w}_0 = \begin{bmatrix}0 \\ \bm{v}_0 \end{bmatrix},
\quad
\widehat{\bm{g}} = \begin{bmatrix}0 \\ \bm{g} \end{bmatrix},
\qquad
\mathcal{H}_m = 
\begin{bmatrix}
0 & -I \\ H_m & 0  
\end{bmatrix}.
$$  
Then $w(t)$ solves an IVP
$$
w'(t) = -\mathcal{H}_m w(t) + \widehat{\bm{g}}, \quad w(0)=\bm{w}_0
$$
and, by the variation of constants formula (see, e.g., 
\cite[Chapter~I, Section~2.3]{HundsdorferVerwer2007}),
\begin{equation}
\label{wt}
w(t) = \bm{w}_0 + t\varphi(-t\mathcal{H}_m)(\widehat{\bm{g}}-\mathcal{H}_m\bm{w}_0).  
\end{equation}
Let $\hat{\omega}_m\in\Rr$ be the smallest eigenvalue of the
matrix $\frac12(\mathcal{H}_m+\mathcal{H}_m^T)$, i.e.,
$$
\Re (\bm{x}^H\mathcal{H}_m\bm{x})\geqs \hat{\omega}, \quad \forall x\in\Cc^{2m}.
$$
Note that the eigenvalues of $\frac12(\mathcal{H}_m+\mathcal{H}_m^T)$ form
pairs $\pm\sigma_i$, where $\sigma_i\geqs 0$, $i=1,\dots,m$, are the singular values
of $\frac12(H_m-I)$ (this can be shown by using the singular value decomposition 
of $H_m-I$).  
Therefore, $\hat{\omega}_m=-\frac12\|\mathcal{H}_m+\mathcal{H}_m^T\|=-\frac12\|H_m-I\|$.  
Then it holds for $t\geqs 0$
$$
t\|\varphi(-t\mathcal{H}_m)\|\leqs t\varphi(-t\hat{\omega}_m).  
$$
Taking the norm of the first $m$ vector components of both
sides in~\eqref{wt}, denoting $\gamma=\|\widehat{\bm{g}}-\mathcal{H}_m\bm{w}_0\|=\sqrt{\|\bm{v}_0\|^2 + \|\bm{g}\|^2}$,
and taking into account that $u(0)=0$, we obtain
$$
\|u(t)\|\leqs t\gamma\|\varphi(-t\mathcal{H}_m)\|\leqs t\gamma\varphi(-t\hat{\omega}_m),
$$
which yields~\eqref{ut_est}.
\end{proof}

From here on, we use upper indices ${}^{(\psi)}$ or ${}^{(\sigma)}$ to distinguish the basis 
vectors and Hessenberg matrices corresponding to the Arnoldi approximations~\eqref{Gal_proj_psi}
and~\eqref{Gal_proj_sigma}.

\begin{proposition}\label{pro:rm_est}
Let $\bm{y}_m(t)$ be the approximate Krylov subspace solution~\eqref{ym}
of IVP~\eqref{ivp} and let
$H_m^{(\psi)}\in\Rr^{m\times m}$, $h_{m+1,m}^{(\psi)}\geqs 0$ and  
$H_m^{(\sigma)}\in\Rr^{m\times m}$, $h_{m+1,m}^{(\sigma)}\geqs 0$ 
be defined by the Arnoldi decompositions~\eqref{eq:arnoldi_decomposition}
obtained for computing $\bm{y}^{(\psi)}(t)$ and $\bm{y}^{(\sigma)}(t)$,
respectively.
Then for the residual~$\bm{r}_m(t)$ of the approximate
solution~$\bm{y}_m(t)$, it holds
\begin{equation}
\label{rm_est}
\begin{aligned}
\|\bm{r}_m(t)\| &\leqs 
h_{m+1,m}^{(\psi)}t\varphi(-t\hat{\omega}^{(\psi)})\|-A\bm{u}+\bm{g}\|+
h_{m+1,m}^{(\sigma)}t\varphi(-t\hat{\omega}^{(\sigma)})\|\bm{v}\|
\\
&\leqs t\varphi(-t\hat{\omega})
\left(
h_{m+1,m}^{(\psi)}\|-A\bm{u}+\bm{g}\|+ h_{m+1,m}^{(\sigma)}\|\bm{v}\|
\right),
\end{aligned}
\end{equation}
where 
$\hat{\omega}^{(\psi)}=-\frac12\|H_m^{(\psi)}-I\|$,
$\hat{\omega}^{(\sigma)}=-\frac12\|H_m^{(\sigma)}-I\|$, and
$\hat{\omega}=\min\{ \hat{\omega}^{(\psi)}, \hat{\omega}^{(\sigma)}\}$.  
Thus, for any $\varepsilon>0$ there exists $\delta>0$ such that
$\|\bm{r}_m(t)\|\leqs \varepsilon$ for $t\in[0,\delta]$.
\end{proposition}

\begin{proof}
Considering the splitting~\eqref{rm_split} of the residual~$\bm{r}_m(t)$ 
into the $\psi$ and $\sigma$ parts
and the fact that both the parts satisfy~\eqref{rm} with $u(t)$ defined 
either in~\eqref{Gal_proj_psi} or in~\eqref{Gal_proj_sigma},
we obtain
\begin{align*}
\|\bm{r}_m(t)\| &\leqs 
h_{m+1,m}^{(\psi)}|\bm{e}_m^Tu^{(\psi)}(t)| + h_{m+1,m}^{(\sigma)}|\bm{e}_m^Tu^{(\sigma)}(t)|
\\
&\leqs h_{m+1,m}^{(\psi)}\|u^{(\psi)}(t)\| + h_{m+1,m}^{(\sigma)}\|u^{(\sigma)}(t)\|.
\\
\intertext{Application of the estimate~\eqref{ut_est} leads to}
\|\bm{r}_m(t)\| &\leqs
h_{m+1,m}^{(\psi)}t\varphi(-t\hat{\omega}^{(\psi)})\|-A\bm{u}+\bm{g}\|+
h_{m+1,m}^{(\sigma)}t\varphi(-t\hat{\omega}^{(\sigma)})\|\bm{v}\|
\\
&\leqs t\varphi(-t\hat{\omega})
\left(
h_{m+1,m}^{(\psi)}\|-A\bm{u}+\bm{g}\|+ h_{m+1,m}^{(\sigma)}\|\bm{v}\|
\right).
\end{align*}
Here $\hat{\omega}=\min\{ \hat{\omega}^{(\psi)}, \hat{\omega}^{(\sigma)}\}$.
\end{proof}

The estimate~\eqref{rm_est}, showing an exponential growth in the residual
norm bound, appears to be pessimistic in practice.  In fact, for symmetric positive 
definite $A$, a much better estimate can be obtained, with a linear growth in the bound.

\begin{proposition}\label{pro:rm_est2}
Let $\bm{y}_m(t)$ be the approximate Krylov subspace solution~\eqref{ym}
of IVP~\eqref{ivp}, with symmetric positive definite $A\in\Rrnn$.
Then for the residual~$\bm{r}_m(t)$ of the approximate
solution~$\bm{y}_m(t)$, it holds
\begin{equation}
\label{rm_est2}
\|\bm{r}_m(t)\| \leqs \left(\frac{h_{m+1,m}^{(\psi)}\beta^{(\psi)}}{\lambda_{\min}^{(\psi)}} + h_{m+1,m}^{(\sigma)}\beta^{(\sigma)}\right) \cdot t
\end{equation}
where $\beta^{(\psi)}=\|\bm{g}-A\bm{u}\|$, $\beta^{(\sigma)}=\|\bm{v}\|$ 
and $\lambda^{(\psi)}_{\min} > 0$ denotes the smallest eigenvalue of $H_m^{(\psi)}$. Thus, for any $\varepsilon>0$ there exists $\delta>0$ such that
$\|\bm{r}_m(t)\|\leqs \varepsilon$ for $t\in[0,\delta]$.
\end{proposition}

\begin{proof}
The solutions $u^{(\psi)}(t)$ of IVP~\eqref{Gal_proj_psi} and
$u^{(\sigma)}(t)$ of IVP~\eqref{Gal_proj_sigma}
can be expressed as
$$
u^{(\psi)}(t)=\frac{t^2}2\psi(-t^2H_m^{(\psi)})\beta^{(\psi)} \bm{e}_1,
\quad
u^{(\sigma)}(t)=t\sigma(-t^2H_m^{(\sigma)})\beta^{(\sigma)} \bm{e}_1.
$$
Denote by $\lambda_i^{(\psi)}, \lambda_i^{(\sigma)}$, $i=1,\dots,m$ the 
eigenvalues of $H_m^{(\psi)}$ and $H_m^{(\sigma)}$, respectively. 
Clearly $\lambda_i^{(\psi)}, \lambda_i^{(\sigma)}$ lie in the spectral 
interval of $A$ and thus, in particular, $\lambda_i^{(\psi)}, 
\lambda_i^{(\sigma)} \geq \lambda_{\min} > 0$. Using the 
standard estimates $|\sin x| \leq x$ and $|1-\cos x| \leq x$ 
for all $x \geq 0$, we therefore have, for $t\geqs 0$,
$$
\begin{gathered}
\frac{t^2}2\|\psi(t^2H_m^{(\psi)})\|\leqs 
\frac{t^2}2 \max_{i=1,\dots,m}|\psi(-t^2\lambda_i^{(\psi)})|
\leqs \min\left\{ \frac{t^2}2, \frac{t}{\lambda_{\min}^{(\psi)}}, \frac{2}{\lambda_{\min}^{(\psi)}} 
\right\}, 
\\
t\|\sigma(t^2H_m^{(\sigma)})\|\leqs t \max_{i=1,\dots,m}|\sigma(t^2\lambda_i^{(\sigma)})|
\leqs\min\left\{ t, \frac{1}{\sqrt{\lambda_{\min}^{(\sigma)}}} \right\}.
\end{gathered}
$$
From the above estimates, we obtain
\begin{align}
\|\bm{r}_m(t)\| &\leqs 
h_{m+1,m}^{(\psi)}|\bm{e}_m^Tu^{(\psi)}(t)| + h_{m+1,m}^{(\sigma)}|\bm{e}_m^Tu^{(\sigma)}(t)|\nonumber
\\
&\leqs h_{m+1,m}^{(\psi)}\|u^{(\psi)}(t)\| + h_{m+1,m}^{(\sigma)}\|u^{(\sigma)}(t)\|\nonumber
\\
&\leqs h_{m+1,m}^{(\psi)}\beta^{(\psi)} \min\left\{ \frac{t^2}2, \frac{t}{\lambda_{\min}^{(\psi)}}, \frac{2}{\lambda_{\min}^{(\psi)}} 
\right\} + h_{m+1,m}^{(\sigma)}\beta^{(\sigma)}\min\left\{ t, \frac{1}{\sqrt{\lambda_{\min}^{(\sigma)}}} \right\},\label{eq:tighter_residual_bound}
\end{align}
from which the assertion follows.
\end{proof}

\begin{remark}\label{rem:tighter_bound}
For better readability, we state the simplified bound~\eqref{rm_est2} in the
assertion of Proposition~\ref{pro:rm_est2}, while the bound~\eqref{eq:tighter_residual_bound}
appearing in the proof is actually tighter and better captures the actual behavior
with respect to $t$. In particular~\eqref{eq:tighter_residual_bound} shows that 
the residual norm is bounded from above and cannot increase to arbitrarily large values.
\end{remark}

Based on the observations stated in Propositions~\ref{pro:rm_est} and~\ref{pro:rm_est2}, 
we now formulate  RT restarting algorithms for IVP~\eqref{ivp} in Algorithms~\ref{alg:second_order_RT_sim} and \ref{alg:second_order_RT_seq}. The algorithms take $A, \bm{u}, \bm{g}$ and $\bm{v}$ as inputs, as well as a final time $t$ for which we want to compute the solution of~\eqref{ivp}, a maximum Krylov subspace dimension $m_{\max}$ and a tolerance $\tol$. We now give details concerning important parts of the algorithms.

\begin{algorithm}
\caption{\label{alg:second_order_RT_sim}RT restarting for second-order IVPs (simultaneous version)}
\begin{algorithmic}[1]\onehalfspacing
\Statex \textbf{Input:} $A \in \mathbb{R}^{n \times n}, \bm{u}, \bm{g},  \bm{v} \in \mathbb{R}^n$, $t > 0$, $m_{\max}$, $\tol$
\Statex \textbf{Output:} Approximate solution $\bm{y}$ of the IVP at time $t$
\State Set \texttt{converged} $\leftarrow$ \texttt{false},\qquad $\bm{y} \leftarrow \bm{u}$
\While{not \texttt{converged}}
\State Set $\widetilde{\bm{g}} \leftarrow -A\bm{y}+\bm{g}$,\qquad $\beta^{(\psi)} \leftarrow \|\widetilde{\bm{g}}\|$,\qquad $\bm{v}_1^{(\psi)} \leftarrow \widetilde{\bm{g}}/\beta^{(\psi)}$
\State Set $\beta^{(\sigma)} \leftarrow \|\bm{v}\|$, \qquad $\bm{v}_1^{(\sigma)} \leftarrow \bm{v}/\beta^{(\sigma)}$
\For{$m=1,\dots,m_{\max}$}
    \State Compute next basis vectors $\bm{v}^{(\psi)}_{m+1}, \bm{v}^{(\sigma)}_{m+1}$ and next columns of $H^{(\psi)}_m, H^{(\sigma)}_m$ \Statex \quad\quad\quad\quad by $m$th step of the Arnoldi/Lanczos process.
    \If{$\max_{s\in[0,t]} \|\bm{r}_m(s)\|\leqs \tol \cdot (\|\bm{g}-A\bm{u}\| +\|\bm{v}\|)$}\label{line:rescheck_sim1}
        \State Set \texttt{converged} $\leftarrow$ \texttt{true},\qquad $\delta \leftarrow t$
        \State break\Comment{leave the \textbf{for}-loop}
    \ElsIf{$m = m_{\max}$}
        \State Find largest $\delta$ such that $\max_{s\in[0,\delta]}\|\bm{r}_m(s)\| \leq \tol$\label{line:rescheck_sim2}
        \EndIf
\EndFor
\State Set $\bm{y}^{(\psi)} \leftarrow \frac12\beta^{(\psi)}\delta^2V_m^{(\psi)}\psi(\delta^2H_m^{(\psi)})\bm{e}_1$, $\bm{v}^{(\psi)} \leftarrow \beta^{(\psi)}\delta V_m^{(\psi)}\sigma(\delta^2H_m^{(\psi)})\bm{e}_1$
\State Set $\bm{y}^{(\sigma)} \leftarrow \beta^{(\sigma)}V_m^{(\sigma)}\sigma(\delta^2 H_m^{(\sigma)})\bm{e}_1$, \qquad $\bm{v}^{(\sigma)} \leftarrow \beta^{(\sigma)} V_m^{(\sigma)}(I-\frac12\delta^2 \psi(\delta^2 H_m^{(\sigma)}))\bm{e}_1$
\State Set $\bm{y} \leftarrow \bm{y} + \bm{y}^{(\psi)} + \bm{y}^{(\sigma)}$, \qquad $\bm{v} \leftarrow \bm{v} + \bm{v}^{(\psi)} + \bm{v}^{(\sigma)}$
\State $t \leftarrow t - \delta$

\EndWhile

\end{algorithmic}
\end{algorithm}

\paragraph{Simultaneous vs.\ sequential solution of two underlying IVPs}
Algorithm~\ref{alg:second_order_RT_sim} is a straightforward adaptation of the RT restarting procedure~\cite{BotchevKnizhnerman2020,BotchevKnizhnermanTyrtyshnikov2020} to the computation of~\eqref{ym}, where we choose a time step $\delta$ based on the overall residual $\bm{r}_m(s)$ of the IVP~\eqref{ivp}. Doing so requires computing approximations for the actions of the $\psi$ and $\sigma$ functions in~\eqref{ym}--\eqref{vm} simultaneously. In particular, both Krylov bases $V_m^{(\psi)}$ and $V_m^{(\sigma)}$ have to be kept in memory at the same time. Thus, in a limited memory setting, in which only $m_{\max}$ Krylov vectors can be stored at the same time, this amounts to reducing the number of Krylov steps to $m_{\max}/2$. As shorter restart lengths typically lead to a higher overall number of matrix vector products, this reduction is clearly undesirable.

We therefore also take a different approach here, depicted in Algorithm~\ref{alg:second_order_RT_seq}. In this approach, we compute the bases $V_m^{(\psi)}$ and $V_m^{(\sigma)}$ sequentially. After choosing an appropriate step size $\delta$ based on $\|\bm{r}_m^{(\psi)}(s)\|$, we compute the corresponding matrix function actions $\bm{y}^{(\psi)}(\delta), \bm{v}^{(\psi)}(\delta)$, allowing us to then discard $V_m^{(\psi)}$ after line~\ref{line:discard_after_here} before starting the Krylov iteration that generates $V_m^{(\sigma)}$. This sequential approach has the drawback that it might happen that the step size $\delta$ is too large for the second iteration, so that $\max_{s\in[0,\delta]}\|\bm{r}_{m_{\max}}^{(\sigma)}(s)\| > \tol^{(\sigma)}$. In that case, we need to reduce the time step accordingly to $\delta^\ast < \delta$. As $V_m^{(\psi)}$ has already been discarded from memory at this point, we then need to recompute it in order to form the approximations $\bm{y}^{(\psi)}(\delta^\ast), \bm{v}^{(\psi)}(\delta^\ast)$, leading to additional matrix vector products. 
% L+M, 2023-02-08:
However, in our experience, 
it is almost always the case that $\delta^*\geqslant\delta$, so that the benefit of 
% L+M, 2023-02-08:
not storing the two bases simultaneously definitely outweighs the need of occasionally recomputing some basis vectors; cf.\ also the numerical experiments reported in Section~\ref{sec:experiments}.

\begin{remark}
A natural alternative to the presented simultaneous RT restarted Algorithm~\ref{alg:second_order_RT_sim}
is to compute the two involved matrix functions employing a single block Krylov subspace.  This can be done, for instance, within the framework of the exponential
block Krylov (EBK) methods, see~\cite[Section~3.3]{Botchev2013}.
Our experiments show that in general, when the vectors on which the matrix functions are evaluated are far from being linearly dependent, convergence behavior of this block Krylov subspace method
is very similar to that of the simultaneous RT restarted Algorithm~\ref{alg:second_order_RT_sim}.
Since, as we will see below, the simultaneous Algorithm~\ref{alg:second_order_RT_sim} turns out to be the least efficient among the methods discussed and tested in this paper,
we do not pay further attention to the block method.
\end{remark}

\paragraph{Stopping criterion}
In practice it is often natural to require that the residual is small
relative to the norm of the initial Krylov subspace vector.  
In our case there are two Krylov subspaces involved (cf.~\eqref{ym}), with the starting vectors
$\bm{g}-A\bm{u}$ and $\bm{v}$, respectively.  Therefore, it is natural
to require that the residual~\eqref{rm_split} of the 
approximation~$\bm{y}_m(t)$ in~\eqref{ym} satisfies
\begin{equation}
\label{stop}
\|\bm{r}_m(t)\|\leqs \tol \cdot (\|\bm{g}-A\bm{u}\| +\|\bm{v}\|).
\end{equation}
In the sequential version of RT restarting, Algorithm~\ref{alg:second_order_RT_seq}, we do not have the overall IVP residual $\bm{r}_m(t)$ available and instead have to work with $\bm{r}_m^{(\psi)}(t),  \bm{r}_m^{(\sigma)}(t)$ individually. Taking into account that $\bm{r}_m(t)=\bm{r}_m^{(\psi)}(t) + \bm{r}_m^{(\sigma)}(t)$, condition~\eqref{stop} holds
\begin{algorithm}[H]
\caption{\label{alg:second_order_RT_seq}RT restarting for second-order IVPs (sequential version)}
\begin{algorithmic}[1]\onehalfspacing

\Statex \textbf{Input:} $A \in \mathbb{R}^{n \times n}, \bm{u}, \bm{g},  \bm{v} \in \mathbb{R}^n$, $t > 0$, $m_{\max}$, $\tol$
\Statex \textbf{Output:} Approximate solution $\bm{y}$ of the IVP at time $t$
\State Set \texttt{converged} $\leftarrow$ \texttt{false},\qquad $\bm{y} \leftarrow \bm{u}$
\State Set $\tol^{(\psi)} \leftarrow \frac12 \tol \left(1 +\frac{\|\bm{v}\|}{\|\bm{g}-A\bm{u}\|}\right)$,\qquad $\tol^{(\sigma)} \leftarrow \frac12 \tol\left(1 + \frac{\|\bm{g}-A\bm{u}\|}{\|\bm{v}\|}\right).$
\While{not \texttt{converged}}
\State Set $\widetilde{\bm{g}} \leftarrow -A\bm{y}+\bm{g}$,\qquad $\beta^{(\psi)} \leftarrow \|\widetilde{\bm{g}}\|$,\qquad $\bm{v}_1^{(\psi)} \leftarrow \widetilde{\bm{g}}/\beta^{(\psi)}$
\For{$m=1,\dots,m_{\max}$}
    \State Compute next basis vector $\bm{v}^{(\psi)}_{m+1}$ and next column of $H^{(\psi)}_m$ by $m$th step of \Statex \quad\quad\quad\quad the Arnoldi/Lanczos process.
    \If{$\max_{s\in[0,t]} \|\bm{r}_m^{(\psi)}(s)\| \leq \tol^{(\psi)}$}\label{line:rescheck_psi}
        \State Set \texttt{converged} $\leftarrow$ \texttt{true},\qquad $\delta \leftarrow t$
        \State break\Comment{leave the \textbf{for}-loop}
    \ElsIf{$m = m_{\max}$}
        \State Find largest $\delta$ such that $\max_{s\in[0,\delta]}\|\bm{r}_m^{(\psi)}(s)\| \leq \tol^{(\psi)}$\label{line:rescheck2_psi}
        \EndIf
\EndFor
\State Set $\bm{y}^{(\psi)} \leftarrow \frac12\beta^{(\psi)}\delta^2V_m^{(\psi)}\psi(\delta^2H_m^{(\psi)})\bm{e}_1$, \qquad $\bm{v}^{(\psi)} \leftarrow \beta^{(\psi)}\delta V_m^{(\psi)}\sigma(\delta^2H_m^{(\psi)})\bm{e}_1$\label{line:discard_after_here}

\Statex\Comment{At this point, $V_m^{(\psi)}$ can be discarded}
\State Set $\beta^{(\sigma)} \leftarrow \|\bm{v}\|$, \qquad $\bm{v}_1^{(\sigma)} \leftarrow \bm{v}/\beta^{(\sigma)}$
\For{$m=1,\dots,m_{\max}$}
    \State Compute next basis vector $\bm{v}^{(\sigma)}_{m+1}$ and next column of $H^{(\sigma)}_m$ by $m$th step of \Statex \quad\quad\quad\quad the Arnoldi/Lanczos process.
    \If{$\max_{s\in[0,\delta]} \|\bm{r}_m^{(\sigma)}(s)\| \leq \tol^{(\sigma)}$}\label{line:rescheck_sigma}
            \State break\Comment{leave the \textbf{for}-loop}
    \ElsIf{$m = m_{\max}$}
        \State Find largest $\delta^\ast$ such that $\max_{s\in[0,\delta^\ast]}\|\bm{r}_m^{(\sigma)}(s)\| \leq \tol^{\sigma}$\label{line:rescheck2_sigma}
        \State Recompute $\bm{y}^{(\psi)}, \bm{v}^{(\psi)}$ for time $\delta^\ast$\label{line:recompute} \Comment{See text for details}
        \State Set $\delta \leftarrow \delta^\ast$
    \EndIf
\EndFor
\State Set $\bm{y}^{(\sigma)} \leftarrow \beta^{(\sigma)}V_m^{(\sigma)}\sigma(\delta^2 H_m^{(\sigma)})\bm{e}_1$, \qquad $\bm{v}^{(\sigma)} \leftarrow \beta^{(\sigma)} V_m^{(\sigma)}(I-\frac12\delta^2 \psi(\delta^2 H_m^{(\sigma)}))\bm{e}_1$
\Statex\Comment{Put together solution for current time step}
\State Set $\bm{y} \leftarrow \bm{y} + \bm{y}^{(\psi)} + \bm{y}^{(\sigma)}$, \qquad $\bm{v} \leftarrow \bm{v} + \bm{v}^{(\psi)} + \bm{v}^{(\sigma)}$\label{line:update_yv}
\State $t \leftarrow t - \delta$

\EndWhile

\end{algorithmic}
\end{algorithm}%
\noindent provided 
\begin{align*}
\|\bm{r}_m^{(\psi)}(t)\| &\leqs \tol^{(\psi)} \|\bm{g}-A\bm{u}\|
\leqs \frac12 \tol \cdot (\|\bm{g}-A\bm{u}\| +\|\bm{v}\|), \quad\text{and}
\\
\|\bm{r}_m^{(\sigma)}(t)\| &\leqs \tol^{(\sigma)} \|\bm{v}\|
\leqs \frac12 \tol \cdot (\|\bm{g}-A\bm{u}\| +\|\bm{v}\|),
\end{align*}
with $\tol^{(\psi)}$ and $\tol^{(\sigma)}$ being the tolerances used to 
compute $\bm{y}_m^{(\psi)}(t)$ and $\bm{y}_m^{(\sigma)}(t)$, respectively. 
Hence, \eqref{stop} is satisfied if we set
\begin{equation}
\label{tols}
\tol^{(\psi)}:= \frac12 \tol \cdot \left(1 +\frac{\|\bm{v}\|}{\|\bm{g}-A\bm{u}\|}\right),
\quad
\tol^{(\sigma)} := \frac12 \tol \cdot \left(1 + \frac{\|\bm{g}-A\bm{u}\|}{\|\bm{v}\|}\right).
\end{equation}
Thus, in practice for a given tolerance $\tol$ the Krylov subspace 
approximations $\bm{y}_m^{(\psi)}(t)$ and $\bm{y}_m^{(\sigma)}(t)$ should be
computed with tolerances given by~\eqref{tols}.
This means that in case $\|\bm{v}\|/\|\bm{g}-A\bm{u}\| \gg 1$ or $\|\bm{v}\|/\|\bm{g}-A\bm{u}\| \ll 1$
either of the two tolerances gets a relaxed value compared to $\tol$, providing a reduction in computational costs.

\paragraph{Step size selection}
In order to monitor the residual norm in lines~\ref{line:rescheck_sim1} and~\ref{line:rescheck_sim2} of Algorithm~\ref{alg:second_order_RT_sim} and lines~\ref{line:rescheck_psi}, \ref{line:rescheck2_psi}, \ref{line:rescheck_sigma} and~\ref{line:rescheck2_sigma} of Algorithm~\ref{alg:second_order_RT_seq}, we follow 
closely the strategy outlined in~\cite{BotchevKnizhnerman2020,BotchevKnizhnermanTyrtyshnikov2020} which has 
proven to be effective in practice. In line~\ref{line:rescheck_sim1} of Algorithm~\ref{alg:second_order_RT_sim} and lines~\ref{line:rescheck_psi} or~\ref{line:rescheck_sigma} of Algorithm~\ref{alg:second_order_RT_seq}, we compute the values 
$\|\bm{r}_m^{(\psi|\sigma)}(s)\|$ for 
$s\in \{\frac{t}{6}, \frac{t}{3}, \frac{t}{2}, \frac{2t}{3}, \frac{5t}{6}, t\}$ which are readily available 
after solving the corresponding projected IVPs of dimension $m$ (cf.\ Table~\ref{t:res}) and check whether they all lie below $\tol^{(\psi|\sigma)}$.

For determining the value of $\delta$ in line~\ref{line:rescheck_sim2} of Algorithm~\ref{alg:second_order_RT_sim} and line~\ref{line:rescheck2_psi} or~\ref{line:rescheck2_sigma} of Algorithm~\ref{alg:second_order_RT_seq}, we use a finer resolution: 
We trace the residual norms $\|\bm{r}_m^{(\psi|\sigma)}(s)\|$ for $s \in \{\Delta t, 2\Delta t, 3\Delta t, \dots \}$
with
$$\Delta t := \frac{t}{2^k \cdot 100},$$
where $k \in \mathbb{N}_0$ is chosen as small as possible such that 
$\|\bm{r}_m^{(\psi|\sigma)}(\Delta t)\| \leq \tol^{(\psi|\sigma)}$. Precisely, we start by 
computing $\|\bm{r}_m^{(\psi|\sigma)}(\frac{t}{100})\|$ and check whether this 
value lies below the required tolerance. If it exceeds the tolerance, we iteratively halve it until we 
find a $\Delta t$ for which the residual satisfies the tolerance. We then successively compute the residual 
norms for $s = 2\Delta t, 3\Delta t, 4\Delta t, \dots$ until we find the first value of $s$ for which the 
tolerance is exceeded.

\subsection{The Gautschi cosine scheme}\label{subsec:gautschi}

We now outline how we can potentially reduce the number of matrix vector products compared to application of Algorithms~\ref{alg:second_order_RT_sim} or~\ref{alg:second_order_RT_seq} by employing the 
Gautschi cosine scheme (cf.~\cite{Gautschi1961,HochbruckLubich1999,HochbruckLubichSelhofer1998,Gautschi2006}) for solving~\eqref{ivp}.

This scheme is based on the observation that for any time step $\delta$, the solution of~\eqref{ivp} satisfies
\begin{equation}\label{eq:gautschi1}
\bm{y}(t+\delta) - 2\bm{y}(t) + \bm{y}(t-\delta) = \delta^2\psi(\delta^2A)(-A\bm{y}(t)+\bm{g}),
\end{equation}
which follows from the variation-of-constants formula. Fixing a step size $\delta$ and denoting $\bm{y}_k = \bm{y}(k\delta)$, 
equation~\eqref{eq:gautschi1} turns into the time stepping scheme
\begin{equation}\label{eq:gautschi2}
\bm{y}_{k+1} = \delta^2\psi(\delta^2A)(-A\bm{y}_k+\bm{g}) + 2\bm{y}_k - \bm{y}_{k-1}
\end{equation}
with $\bm{y}_0 = \bm{u}$ and $\bm{y}_1=\bm{y}(\delta)$ computed
by~\eqref{yex}.
The iteration~\eqref{eq:gautschi2} can be further rewritten to obtain a one-step version, cf.~\cite[Equation~(10)]{HochbruckLubich1999},
\begin{eqnarray}
\bm{v}_{k+1/2} &=& \bm{v}_k + \frac12\delta\psi(\delta^2A)(-A\bm{y}_k + \bm{g}),\nonumber\\
\bm{y}_{k+1}   &=& \bm{y}_k + \delta\bm{v}_{k+1/2},\label{eq:gautschi_onestep}\\
\bm{v}_{k+1}   &=& \bm{v}_{k+1/2} + \frac12\delta \psi(\delta^2 A)(-A\bm{y}_{k+1} + \bm{g}),\nonumber
\end{eqnarray}
with starting vectors $\bm{y}_0 = \bm{u}$ and 
\begin{equation}\label{eq:gautschi_starting_vector}
\bm{v}_0 = \sigma(\delta^2 A)\bm{v}.
\end{equation}

It is important to note that in each iteration 
of~\eqref{eq:gautschi_onestep} (except the first)
the action of the $\psi$ function required to compute $\bm{v}_{k+1/2}$
has already been computed in the previous iteration. Also note that 
the values $\bm{v}_k$ computed in~\eqref{eq:gautschi_onestep} do 
\emph{not} correspond to the values $\bm{y}^\prime(\delta k)$, but 
can rather be interpreted as \emph{averaged} velocities over the 
interval $[(k-1)\delta, (k+1)\delta]$.

Of course, in practice the matrix functions appearing in~\eqref{eq:gautschi_onestep} 
and~\eqref{eq:gautschi_starting_vector} are not computed exactly, but are replaced by
their Galerkin approximations~\eqref{Gal_proj_psi} and~\eqref{Gal_proj_sigma}, respectively.
Then, assuming that Algorithm~\ref{alg:second_order_RT_seq} chooses the same step size $\delta$ as used
in the Gautschi scheme, the number of matrix vector products needed for the 
Gautschi scheme is roughly half of the number of matrix vector products needed for computing
the solution via~\eqref{ym}: The number of $\psi$ function evaluations is the same, 
but $\sigma$ needs to be approximated only once in the Gautschi scheme for computing 
the starting vector $\bm{v}_0$; see~\eqref{eq:gautschi_starting_vector}.

\begin{algorithm}
\caption{\label{alg:gautschi}Gautschi cosine scheme with residual-based step size selection}
\begin{algorithmic}[1]\onehalfspacing
\Statex \textbf{Input:} $A \in \mathbb{R}^{n \times n}, \bm{u}, \bm{v}, \bm{g} \in \mathbb{R}^n$, $t > 0$, $m$, $\tol$, safety factor $0 < \alpha < 1$
\Statex \textbf{Output:} Approximate solution $\bm{y}$ of the IVP at time $t$
\Statex \Comment{Initial computation of $\bm{v}_0$ with step size selection}
\State Set $\widetilde{m} \leftarrow \lfloor\alpha\cdot m\rfloor$
\State Perform $\widetilde{m}$ Arnoldi/Lanczos steps for $A$ and $\bm{v}$, yielding $\beta_m^{(\sigma)}, V_m^{(\sigma)}, H_m^{(\sigma)}$
\State Find largest $\delta$ such that $\max_{s\in[0,\delta]}\|\bm{r}_m^{(\sigma)}(s)\| \leq \tol$
\State Set $\bm{v}_0 \leftarrow \beta_m^{(\sigma)}V_m^{(\sigma)}\sigma(\delta^2 H_m^{(\sigma)})\bm{e}_1$
\State Perform $\widetilde{m}$ Arnoldi/Lanczos steps for $A$ and $(-A\bm{u}+\bm{g})$, yielding $\beta_m^{(\psi)}, V_m^{(\psi)}, H_m^{(\psi)}$
\If{$\|\bm{r}_m^{(\psi)}(\delta)\| > \tol$} \Comment{Update step size and recompute $\sigma$ action}
\State Find largest $\delta^\ast$ such that $\max_{s\in[0,\delta^\ast]}\|\bm{r}_m^{(\psi)}(s)\| \leq \tol$
\State Perform $\widetilde{m}$ Arnoldi/Lanczos steps for $A$ and $\bm{v}$, yielding $\beta_m^{(\sigma)}, V_m^{(\sigma)}, H_m^{(\sigma)}$ 
\Statex\Comment{only if $V_m^{(\sigma)}$ was discarded from memory}
\State $\bm{v}_0 \leftarrow \beta_m^{(\sigma)}V_m^{(\sigma)}\sigma((\delta^\ast)^2 H_m^{(\sigma)})\bm{e}_1$
\State Set $\delta \leftarrow \delta^\ast$
\EndIf
\State \texttt{steps} $\leftarrow \lceil\frac{t}{\delta}\rceil$
\State $\bm{x} \leftarrow \frac12\beta_m^{(\psi)}\delta V_m^{(\psi)}\psi(\delta^2 H_m^{(\psi)})\bm{e}_1$
\Statex \Comment{Actual time stepping}
\For{$k=0,\dots,\texttt{steps-1}$}
\State $\bm{v}_{k+1/2} \leftarrow \bm{v}_k + \bm{x}$
\State $\bm{y}_{k+1} \leftarrow \bm{y}_k + \delta \bm{v}_{k+1/2}$
\State Perform $m$ Arnoldi/Lanczos steps for $A$ and $-A\bm{y}_{k+1}+\bm{g}$, yielding $\beta_m, V_m, H_m$
\State Find largest $\tilde{\delta} \in [0,\delta]$ such that $\max_{s\in[0,\delta^\ast]}\|\bm{r}_m^{(\psi)}(s)\| \leq \tol$
\If{$\tilde{\delta} = \delta$}
\Comment{No time step repair necessary}
\State $\bm{x} \leftarrow \frac12\beta_m\delta V_m\psi(\delta^2 H_m)\bm{e}_1$ \Comment{stored temporarily, as needed twice}
\Else
\Comment{Completed step too small $\Rightarrow$ time step repair}
\State $\tilde{\bm{x}} \leftarrow \frac12\beta_m\tilde{\delta} V_m\psi(\tilde{\delta}^2 H_m)\bm{e}_1$
\State $\tilde{\bm{v}} \leftarrow \beta_m\tilde{\delta} V_m\sigma(\tilde{\delta}^2H_m)\bm{e}_1$
\State Solve IVP~\eqref{ivp} on the time interval $(0, \delta-\tilde{\delta})$ with initial values $\tilde{\bm{x}}, \tilde{\bm{v}}$ via
\Statex \quad\quad\quad\quad any RT scheme, yielding $\bm{x}$
\EndIf
\State $\bm{v}_{k+1} \leftarrow \bm{v}_{k+1/2} + \bm{x}$
\EndFor

\end{algorithmic}
\end{algorithm}

The difficulty in most efficiently employing the Gautschi cosine scheme lies in choosing
a good step size $\delta$ (which needs to stay fixed throughout the time stepping process). 
Note that in our case of a constant inhomogeneity $\bm{g}$, the scheme~\eqref{eq:gautschi_onestep} actually yields the \emph{exact} solution of~\eqref{ivp}, i.e., $\bm{y}_{m} = \bm{y}(m\delta)$, if the $\psi$ and $\sigma$ functions are evaluated exactly (while it can be shown to have an error of $O(\delta^2)$ for non-constant $\bm{g}$; see~\cite[Theorem~3.1]{HochbruckLubich1999}). Thus, when choosing the step size, our main goal is to ensure that the actions of these functions
can be computed accurately enough with the available memory in order to not spoil the accuracy
of the final approximation.

Here, our residual concept for the $\psi$ and $\sigma$ functions comes in handy. We do 
not fix $\delta$ a priori, but instead determine it based on the residual, when 
evaluating the action of the $\sigma$ function in~\eqref{eq:gautschi_starting_vector}. 
We perform as many Krylov steps as the available memory permits and then choose $\delta$ 
as large as possible such that $\|\bm{r}^{(\sigma)}_{m}(\delta)\| \leq \tol$. 
When first approximating the action of the $\psi$ function in~\eqref{eq:gautschi_onestep}, 
we check whether this choice of $\delta$ also leads to $\|\bm{r}^{(\psi)}_{m}(\delta)\| \leq \tol$ (which will often be the case in practice). If it does, we fix $\delta$ and proceed with the time 
stepping scheme~\eqref{eq:gautschi_onestep}, each time approximating the action of the 
$\psi$ function using $m$ Krylov steps. In the unlikely event that the time step $\delta$
chosen for $\sigma$ is too large for $\psi$, we reduce $\delta$ appropriately and recompute
the action of $\sigma$ for this new value of $\delta$ before starting the time stepping (cf.~also 
the similar discussion for the RT restarting procedure in Section~\ref{subsec:rt_restarting}). Note that if $V_m^{(\sigma)}$ was discarded due to memory constraints, this requires performing another $m$ Krylov steps.

Since the actual convergence behavior of Krylov subspace approximations
depends on a starting vector, it may happen that the residual
in one of the $\psi$ evaluations in~\eqref{eq:gautschi_onestep}
does not satisfy the requirement $\|\bm{r}^{(\psi)}_{m}(\delta)\| \leq \tol$.
In this case we find a largest value $\tilde{\delta}\in(0,\delta)$ 
such that $\|\bm{r}^{(\psi)}_{m}(\tilde{\delta})\| \leq \tol$ holds,
set $\tilde{\bm{x}}(\tilde{\delta}):=\delta \psi(\tilde{\delta}^2 A)(-A\bm{y}_{k} + \bm{g})$
and bridge the time interval $(\tilde{\delta},\delta)$ by solving the IVP~\eqref{ivp} on the time interval $t\in(0,\delta-\tilde{\delta})$
with initial data $\bm{u}=\tilde{\bm{x}}(\tilde{\delta})$,
$\bm{v}=\tilde{\bm{x}}'(\tilde{\delta})$.
This IVP is solved numerically according to relation~\eqref{ym} with two additional matrix functions evaluations.
To reduce the likelihood that the chosen time step is too large for later cycles of the method, one can use a slightly smaller Krylov dimension for the initial evaluations of the $\sigma$ and $\psi$ function, e.g., replacing $m$ by $\lfloor\alpha \cdot m\rfloor$, where $0 < \alpha < 1$ is a ``safety factor'' (we use $\alpha = 0.85$ in our experiments).

We summarize the approach outlined above in Algorithm~\ref{alg:gautschi}.

\subsection{Two-pass Lanczos}\label{sec:2pl}
When $A$ is symmetric, a well-known alternative to restarting the Krylov method is the so-called ``two-pass Lanczos process'' discussed in, e.g.,~\cite{Borici2000,FrommerSimoncini2008,GuettelSchweitzer2021}, which we brief\/ly describe in the following. For symmetric $A$, the Arnoldi method reduces to the three-term recurrence Lanczos process, so that only three Krylov basis vectors need to be stored at a time in order to perform the orthogonalization (at least in exact arithmetic). As forming the Krylov iteration $\bm{y}_m$ still requires all basis vectors (in contrast to the solution of linear systems, where the short recurrence for the basis vectors translates into a short recurrence for the iterates), one proceeds as follows. In a first call to the Lanczos method, one computes the basis vectors by the three-term recurrence (discarding ``old'' basis vectors that are no longer needed) and assembles the (tridiagonal) matrix $H_m$. As soon as a suitable stopping criterion is fulfilled, the ``coefficient vectors''
\begin{equation*}
\bm{x}_m^{(\psi)}(t) = \frac12 \|-A\bm{u}+\bm{g}\|t^2\psi(t^2 H_m)\bm{e}_1, \qquad \bm{x}_m^{(\sigma)}(t) = \|\bm{v}\|t\sigma(t^2H_m)\bm{e}_1, 
\end{equation*}
needed for forming $\bm{y}_m(t)$ are computed. Then, in a second call, the basis vectors are recomputed and used for computing $\bm{y}_m(t)$ by iteratively updating the sum
\begin{equation*}
    \bm{y}_m(t) = \sum_{i=1}^m \left(\bm{x}_i^{(\psi)}(t)\bm{v}_i^{(\psi)} + \bm{x}_i^{(\sigma)}(t)\bm{v}_i^{(\sigma)}\right).
\end{equation*}
This way, available memory is no limitation for the applicability of the Krylov method, at the price of doubling the number of matrix vector products that need to be carried out. 
% MB 2022-06-02:
The possibility to avoid restarting in the Lanczos process by this two-pass construction allows to argue that ``there is absolutely no need to restart''~\cite{TalEzer2007}.  However, as our experiments presented below show, restarting may be more efficient than the two-pass Lanczos process; see also the corresponding discussion in~\cite{GuettelSchweitzer2021}.

In the context of two-pass Lanczos, our residual concept can be used as a problem-related stopping criterion, instead of using more simplistic error estimators (as, e.g., the norm of the difference of two iterates, as it is often done in practical computations).

We now make a few remarks related to this method.
\begin{remark}\label{rem:2pl_roundoff}
It is well-known that the computed Lanczos vectors quickly fail to be orthogonal to each other in the presence of round-off error. This loss of orthogonality typically leads to a delayed convergence compared to the exact method (or a method with explicit reorthogonalization). This observation raises the question whether stopping the iteration based on our residual concept is reasonable in an actual computation, or whether the relation~\eqref{rm} is too spoiled by round-off error, making it unusable (this problem can be expected to be less prominent for a restarted method where the number of iterations is kept low, so that loss of orthogonality is typically not as severe). 

Fortunately, even when orthogonality is lost, the Lanczos decomposition on which formula~\eqref{rm} for the residual is based still holds fairly accurately. Precisely,
\begin{equation}\label{eq:lanczos_decomposition_roundoff}
AV_m = V_mH_m + h_{m+1,m} \bm{v}_{m+1} \bm{e}_m^T + F_m, \text{where } F_m \in \mathbb{R}^{n \times m},\ \|F_m\| \lesssim \varepsilon \|A\|,
\end{equation}
with $\varepsilon$ denoting the unit round-off; see~\cite[Chapter~13, Section~4]{Parlett1998}. Thus, repeating the steps in the derivation of~\eqref{rm}, employing~\eqref{eq:lanczos_decomposition_roundoff} instead of~\eqref{eq:arnoldi_decomposition}, we find that in finite precision arithmetic, it holds
\begin{equation*}
\bm{r}_m(t) = -h_{m+1,m}\bm{v}_{m+1}\bm{e}_m^Tu(t) - F_mu(t),
\end{equation*}
indicating that the residual is not spoiled more than is to be expected by round-off error, so that it can safely be used as stopping criterion.
\end{remark}

\begin{remark}\label{rem:2pl_cost}
When comparing the computational cost of the two-pass Lanczos method to that of a restarted method, one has to keep in mind that the two-pass Lanczos comes with further relevant costs in addition to matrix vector products: checking a stopping criterion (be it residual-based or simply the difference between iterates) requires the evaluation of the $\psi$ and $\sigma$ functions at the tridiagonal matrix $H_m$. If a large number $m$ of iterations is necessary for convergence, frequent computation of these matrix functions induces a non-negligible additional computational cost (even when the stopping criterion is not checked in each iteration). By exploiting, e.g., the MRRR algorithm~\cite{DhillonParlett2004} to solve the tridiagonal eigenproblem for $H_m$, one can evaluate $f(H_m)$ using $\mathcal{O}(m^2)$ flops, which may not seem too problematic unless $m$ becomes really huge. However, in a typical implementation, one will check the stopping criterion regularly, typically after a fixed number of iterations. When the residual is checked each $k$th iteration, then the overall cost for tridiagonal matrix function evaluations \emph{across all iterations} scales as $\mathcal{O}(m^3)$, which certainly becomes noticeable for realistically occurring values of $m$; cf.~also our experiments in Section~\ref{sec:experiments}, in particular Figure~\ref{fig:t_vs_time_anisotropic}.
\end{remark}

\section{Analysis of the residual convergence}\label{sec:residual_convergence}
We now give some estimates on the norms of Krylov subspace residuals in terms of Faber and Chebyshev series. Faber series have been a useful tool for investigating convergence of
the Arnoldi method since~\cite{Eiermann1989}; see also \cite{BeckermannReichel2009,Knizh91}. 
For simplification of formulae we assume throughout this section that 
$\beta^{(\psi)} = \beta^{(\sigma)}=1$.

\subsection{Estimates in terms of the Faber series}

Let $\Phi_j$ be the Faber polynomials (refer to \cite{SuetinPankratiev1998}) 
for $W(A)$, the numerical range of $A$. 

The following assertion is an analogue of \cite[Proposition~1]{BotchevKnizhnerman2020} 
and \cite[Proposition~4.1]{BotchevKnizhnermanTyrtyshnikov2020}.
 
\begin{proposition} \label{P1}
Let 
\be \label{fun1}
f(z;t) := t\sigma(t^2z) = \frac{\sin(t\sqrt{z})}{\sqrt{z}},
\ee
  let us have the Faber series decomposition
\be \label{app3}
f(z;t) = \sum_{j=0}^{+\infty}f_j(t)\Phi_j(z), 
\qquad z\in W(A), \quad t\ge0,
\ee
where $t$ is considered as a parameter, and let $\bm{r}_m(t)$ be the underlying residual. Then
\be \label{appr2}
\|\bm{r}_m(t)\| \le 2\|A\|\sum_{j=m-1}^{+\infty}|f_j(t)|.
\ee
\end{proposition}

\begin{proof}
The superexponential convergence in $j$ and the smoothness of $f_j(t)$ in $t$ enable
one to differentiate the series (\ref{app3}) in $t$. The decomposition (\ref{app3}) then implies the decomposition of the approximant $\bm{y}_m$ as
\be \label{app6}
\bm{y}_m(t)=V_m^{(\sigma)}f(H_m;t) \bm{e}_1=V_m^{(\sigma)}\sum_{j=0}^{\infty} f_j(t)\Phi_j(H_m) \bm{e}_1.
\ee
Since
\be \label{diffeq1}
\frac{\partial^2 f}{\partial t^2} + zf = 0,
\ee
by differentiation we obtain
\bea \label{app1}
0 = z\sum_{j=0}^\infty f_j(t)\Phi_j(z) + \sum_{j=0}^\infty f_j^{\prime\prime}(t)\Phi_j(z)
= \sum_{j=0}^\infty \left[f_j(t)z+f_j^{\prime\prime}(t)\right]\Phi_j(z).
\eea
Exploiting (\ref{app6}) and (\ref{app1}) with $H_m$ substituted for $z$, and the 
equality $\deg\Phi_j=j$, derive
\beas
-\bm{r}_m(t) &=& \bm{y}_m^{\prime\prime}(t)+A\bm{y}_m(t) \\
&=& V_m^{(\sigma)}\sum_{j=0}^\infty f_j^{\prime\prime}(t)\Phi_j(H_m)\bm{e}_1 + AV_m^{(\sigma)}\sum_{j=0}^\infty f_j(t)\Phi_j(H_m)\bm{e}_1 \\
&=& \left( V_m^{(\sigma)}H_m+h_{m+1,m}\bm{v}_{m+1}^{(\sigma)}\bm{e}_m^\tT \right) \sum_{j=0}^\infty f_j(t)\Phi_j(H_m)\bm{e}_1
+ V_m^{(\sigma)}\sum_{j=0}^\infty f_j^{\prime\prime}(t)\Phi_j(H_m)\bm{e}_1 \\
&=& V_m^{(\sigma)}\sum_{j=0}^\infty \left[f_j^{\prime\prime}(t){I}_m+f_j(t)H_m\right]\Phi_j(H_m)\bm{e}_1
+ h_{m+1,m}\bm{v}_{m+1}^{(\sigma)}\bm{e}_m^\tT \sum_{j=0}^\infty f_j(t)\Phi_j(H_m)\bm{e}_1 \\
&=& h_{m+1,m}\bm{v}_{m+1}^{(\sigma)}\bm{e}_m^\tT \sum_{j=m-1}^\infty f_j(t)\Phi_j(H_m)\bm{e}_1,
\eeas
where ${I}_m$ denotes the size $m$ identity matrix.
The bound $\|\Phi_j(H_m)\|\le2$ (see \cite[Theorem~1]{Beckermann2005}) now implies (\ref{appr2}).
\end{proof}

\begin{remark}
Comparing estimate (\ref{appr2}) and the estimate on the error
\[
\|\bm{y}(t)-\bm{y}_m(t)\| \le 4\sum_{j=m}^{+\infty}|f_j(t)|
\]
(see \cite[Theorem~3.2]{BeckermannReichel2009} for the general case and
\cite[Theorem~3]{DruskinKnizhnerman1989} for the symmetric case) shows that 
the two upper bounds mainly differ in coefficients. Thus, we can conjecture
that the error and the residual behave similarly to each other.
\end{remark}

\begin{proposition} \label{P2}
If we replace the function (\ref{fun1}) with the parametric function 
\be \label{fun2}
f(z;t):= \frac12t^2\psi(t^2z) = \frac{1-\cos(t\sqrt{z})}{z}
\ee
in Proposition~\ref{P1}, then estimate (\ref{appr2}) remains valid.
\end{proposition}

\begin{proof}
The proof of this result is given in Appendix~\ref{sec:appendix}.
\end{proof}

\subsection{The symmetric case}

We shall assume here that $A^T=A$ and the spectral interval of $A$ is contained 
in the segment $[0,1]$  (rescaling can easily be done if necessary).
In this case the Faber series for any function which is smooth on $[0,1]$ essentially 
reduces to a series in the shifted (onto $[0,1]$) Chebyshev polynomials $T_k^*$:
\be \label{FabCheb}
\Phi_0=T_0^*, \quad \Phi_k=2T_k^*\quad (k\ge1)
\ee
(see \cite[Chapter~II, Section~1, (21)]{SuetinPankratiev1998} and \cite[Chapter~I, Section~1, (42)]{Paszkowski1983}).

Following \cite{Paszkowski1983}, we shall denote respectively by $a_k[h]$ and 
$a_k^*[h]$ the $k$th plain and shifted Chebyshev coefficient of a function $h$ 
defined on $[-1,1]$ or $[0,1]$, so that
\[
h(x) = \sum_{k=0}^{+\infty}{\!}' a_k[h]T_k(x), \quad -1\le x\le 1,
\quad\mbox{or}\quad
h(x) = \sum_{k=0}^{+\infty}{\!}' a_k^*[h]T_k^*(x), \quad 0\le x\le 1,
\]
where $T_k$ and $T_k^*$ are the $k$th first kind plain or shifted Chebyshev 
polynomials respectively and the prime ${}\prime$ means that the term at $k=0$ is divided 
by two.

\begin{lemma} \label{L1}
The shifted Chebyshev series decomposition
\be \label{chebser1}
\frac{\sin(t\sqrt{z})}{\sqrt{z}}
=4\sum_{k=0}^\infty{\!}'(-1)^k \sum_{l=k}^\infty J_{2l+1}(t)T_k^*(z)
\ee
holds, where $J_k$ are the Bessel functions \cite[Section~9.1]{AbramowitzStegun1964}; that is,
\be \label{resL1}
a_k^*\left[\frac{\sin(t\sqrt{z})}{\sqrt{z}}\right]
=4(-1)^k \sum_{l=k}^\infty J_{2l+1}(t).
\ee
\end{lemma}

\begin{proof}
Setting $y=\sqrt{z}$ and exploiting the formulae
\cite[Chapter~II, Section~11, (21); Chapter~I, Section~1, (44)]{Paszkowski1983}, derive
\beas
\frac{\sin(t\sqrt{z})}{\sqrt{z}} &=& \frac{\sin(ty)}{y}
=4\sum_{k=0}^\infty{\!}'(-1)^k \sum_{l=k}^\infty J_{2l+1}(t)T_{2k}(y) \\
&=&4\sum_{k=0}^\infty{\!}'(-1)^k \sum_{l=k}^\infty J_{2l+1}(t)T_{2k}(\sqrt{z})
=4\sum_{k=0}^\infty{\!}'(-1)^k \sum_{l=k}^\infty J_{2l+1}(t)T_k^*(z).\hfill
\eeas
\phantom{x}
\end{proof}

\begin{proposition} \label{P3}
If $t\le1$ and $m\ge2$, then the residual associated with function~(\ref{chebser1}) is estimated as
\be \label{resP3}
\|\bm{r}_m(t)\| \le 16\cdot\frac{(t/2)^{2m-1}}{(2m-1)!}.
\ee
\end{proposition}

\begin{proof}
The proof of this result is given in Appendix~\ref{sec:appendix}.
\end{proof}

\begin{lemma} \label{L2}
The shifted Chebyshev coefficients of the function (\ref{fun2})
are determined by the equality
\be \label{resL2}
a_k^*\left[\frac{1-\cos(t\sqrt{z})}{z}\right]
={8}(-1)^k\sum_{l=0}^\infty (l+1)J_{2(k+l+1)}(t),
\qquad k\ge0,
\ee
where $J_k$ are the Bessel functions.
\end{lemma}

\begin{proof}
Due to the formulae \cite[Chapter~II, Section~11, (19); Chapter~I, Section~1, (44)]{Paszkowski1983} we have the shifted Chebyshev decomposition
\beas
\cos(t\sqrt{z}) = 2\sum_{k=0}^\infty (-1)^kJ_{2k}(t)T_{2k}(\sqrt{z})
= 2\sum_{k=0}^\infty (-1)^kJ_{2k}(t)T_{k}^*(z).
\eeas
Thus,
\[
a_k^*[1-\cos(t\sqrt{z})] = -2(-1)^kJ_{2k}(t), \qquad k\ge1.
\]
Finally, the formula \cite[Chapter~II, Section~9, (66)]{Paszkowski1983} leads to
\beas
a_k^*\left[\frac{1-\cos(t\sqrt{z})}{z}\right]
&=& 4\sum_{l=0}^\infty (-1)^l\binom{l+1}{1}a_{k+l+1}^*[1-\cos(t\sqrt{z})] \\
&=& 4\sum_{l=0}^\infty (-1)^l(l+1)(-2)(-1)^{k+l+1}J_{2(k+l+1)}(t)\\
&=& 8(-1)^k\sum_{l=0}^\infty (l+1)J_{2(k+l+1)}(t),
\eeas
which concludes the proof.
\end{proof}

\begin{proposition} \label{P4}
If $t\le1$ and $m\ge2$, then the residual associated with function~(\ref{fun2}) is estimated as
\be \label{resP4}
\|\bm{r}_m(t)\| \le \frac{128}{15}\cdot\frac{(t/2)^{2m}}{(2m)!}.
\ee
\end{proposition}

\begin{proof}
The proof of this result is given in Appendix~\ref{sec:appendix}.
\end{proof}

\section{Numerical experiments}\label{sec:experiments}
We now show the results of several numerical experiments in order to demonstrate the practical performance of the different algorithms presented in this paper. All experiments are carried out in MATLAB R2022a on a PC with an AMD Ryzen 7 3700X 8-core CPU with clock rate 3.60GHz and 32 GB RAM. In all experiments, we restrict the maximum Krylov dimension to $m=30$ and if $A$ is symmetric, we use the three-term recurrence Lanczos process without reorthogonalization instead of the full Arnoldi recurrence in all considered algorithms. For computing the action of the $\psi$ and $\sigma$ functions evaluated at the Hessenberg matrix $H_m$, we employ a Schur decomposition $Q^TH_mQ = T$ in conjunction with Parlett's recurrence~\cite{Parlett1976}. When evaluating $\psi$ or $\sigma$ on the diagonal of the upper triangular matrix $T$, we replace the formulas~\eqref{psi_sigma} by their respective $(1,1)$ Pad\'e approximants when the function argument is smaller than $10^{-3}$ in order to increase robustness.

\subsection{3D wave equation}

As a first example, we consider semi-discretizations of the three-dimensional wave equation in $u = u(t,x,y,z)$,
\begin{equation}\label{eq:3d_wave_equation}
\left\{\begin{aligned}
& u_{tt} = k_x u_{xx} + k_y u_{yy} + k_z u_{zz}, \\
& u(0,x,y,z) = u_0(x,y,z), \quad u_t(0,x,y,z) = v_0(x,y,z),
\end{aligned}\right.
\end{equation}
on the domain $\Omega = (0,1) \times (0,1) \times (0,1)$, with coefficients $k_x,k_y,k_z \in \mathbb{R}^+$ and functions $u_0,v_0:\mathbb{R}^3 \longrightarrow \mathbb{R}$ specifying the initial conditions. We discretize the second spatial derivatives by the usual seven-point finite difference stencil on a uniform $n_x \times n_y \times n_z$ grid and impose homogeneous Dirichlet boundary conditions. This results in a system of the form~\eqref{ivp}, where
\begin{equation*}\label{eq:A_wave}
A = k_zL_z \otimes I_y \otimes I_x + I_z \otimes k_yL_y \otimes I_x + I_z \otimes I_y \otimes k_xL_x \in \mathbb{R}^{n_xn_yn_z \times n_xn_yn_z}
\end{equation*}
with $I_i, i \in \{x,y,z\}$ the identiy matrix of size $n_i$ and
\begin{equation}\label{eq:Laplacian_discretization}
L_{i} = \frac{1}{h_i^2}\text{tridiag}(1,-2,1) \in \mathbb{R}^{n_i \times n_i}, \quad h_i = 1/(n_i + 1) \quad \text{ for } i \in \{x,y,z\}
\end{equation}
and the vectors $\bm{u}, \bm{v}$ result from discretization of the initial conditions $u_0, v_0$.

\paragraph{Isotropic case $k_x = k_y = k_z = 1$}
We begin by considering the isotropic case, in which the coefficients $k_x = k_y = k_z$ coincide (allowing to write the wave equation in the more compact, classical form $u_{tt} = c^2 \Delta u$ with $c^2 := k_x$). As initial conditions, we consider the functions $u_0(x,y,z) = (1-x)^3(1-y^2)(1-z^2)$ and $v_0(x,y,z) \equiv 1$. We aim to approximate the solution $u$ at time $t = 1$.

We compare the performance of the RT restarting Lanczos method both in the simultaneous (Algorithm~\ref{alg:second_order_RT_sim}) and the sequential version (Algorithm~\ref{alg:second_order_RT_seq}), the Gautschi cosine scheme with residual-based step size selection (Algorithm~\ref{alg:gautschi}) and the two-pass Lanczos method described in Section~\ref{sec:2pl}. Note that---in order to simulate the situation in a limited memory environment---we restrict the maximum dimension of the Krylov subspaces to $m/2$ in the simultaneous version of the RT scheme, as \emph{two} Krylov bases have to be kept in memory at the same time.

\begin{figure}
\centering
\tikzsetnextfilename{tol_vs_acc_isotropic}
\pgfplotsset{height=0.43\linewidth,width=0.975\linewidth,compat=1.10,every axis/.append style={legend style={/tikz/every even column/.append style={column sep=6pt}}}}
\pgfplotsset{every tick label/.append style={font=\scriptsize}}

\noindent%
\begin{tikzpicture}[scale=1]%
    \begin{loglogaxis}[legend columns=1,cycle list name=list_std, grid=major, 
   	xlabel={\small residual tolerance}, ylabel={\small relative accuracy}, xmin=5e-9,xmax=2e-1,
   	title={\small $40 \times 40 \times 40$ isotropic wave equation},x dir=reverse]

\addplot+[] table [x ={tol},y ={simul}] {figures/data/tol_vs_acc_isotropic.dat};\addlegendentry{\scriptsize  RT Lanczos, simultaneous}
\addplot+[] table [x ={tol},y ={seq}] {figures/data/tol_vs_acc_isotropic.dat};\addlegendentry{\scriptsize  RT Lanczos, sequential}
\addplot+[] table [x ={tol},y ={gautschi}] {figures/data/tol_vs_acc_isotropic.dat};\addlegendentry{\scriptsize  RT Gautschi}
\addplot+[] table [x ={tol},y ={2pl}] {figures/data/tol_vs_acc_isotropic.dat};\addlegendentry{\scriptsize  Two-pass Lanczos}

\end{loglogaxis}
\end{tikzpicture}
\caption{Relative accuracy of the solution for the isotropic wave equation computed by the different algorithms, depending on the residual accuracy that is used.}
\label{fig:tol_vs_acc_isotropic}
\end{figure}
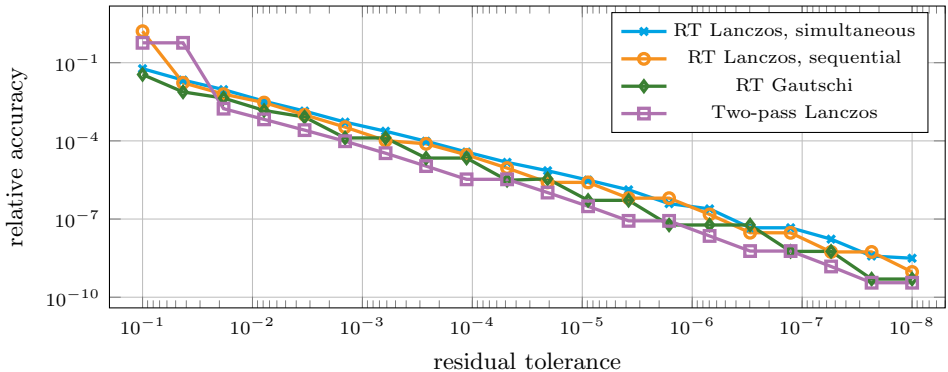

We begin by illustrating how the imposed residual tolerance $\tol$ translates into a (relative) accuracy of the solution, $\|\bm{y}(t)-\widetilde{\bm{y}}(t)\|/\|\bm{y}(t)\|$, where $\bm{y}(t)$ is the exact solution of the IVP and $\widetilde{\bm{y}}(t)$ the solution computed by one of the considered algorithms. To do so, we discretize~\eqref{eq:3d_wave_equation} on a $40 \times 40 \times 40$ grid and run all algorithms with different residual tolerances ranging from $10^{-1}$ to $10^{-8}$. The resulting final accuracies for the different algorithms are shown in Figure~\ref{fig:tol_vs_acc_isotropic}. We observe that the sequential RT Lanczos method, the Gautschi scheme and two-pass Lanczos produce approximations with comparable accuracy (except for very crude tolerances, where the sequential RT Lanczos method and the Gautschi scheme produce less accurate approximations), which mostly lies even a bit below the order of the imposed residual tolerance. The simultaneous RT Lanczos method is the least accurate, but its accuracy still lies in the order of the residual tolerance.

\begin{table}[]
\small
\centering
\caption{3D isotropic wave equation test problem. Number of matvecs required (and in brackets accuracy reached) by the three schemes for different grids and accuracy requirements ($^\star$: The number of matvecs for the two-pass Lanczos method does not fully account for the computational cost of the method, see Remark~\ref{rem:2pl_cost}).}
\label{tab:3dwave_isotropic}
\begin{tabular}{ccccc}
\hline\hline
tolerance & RT simultaneous & RT sequential & RT Gautschi & 2P Lanczos$^\star$ \\
\hline
\multicolumn{5}{c}{$10\times10\times10$ grid}  \\
\hline
{\tt 1e-04}  &     60 ({\tt 1.0e-05})  &       47 ({\tt 7.7e-06})  &     47 ({\tt 7.7e-06})   &    98  ({\tt 7.7e-06})\\
{\tt 1e-06}  &     77 ({\tt 1.9e-07})  &       52 ({\tt 2.3e-08})  &     73 ({\tt 3.7e-08})   &   110  ({\tt 2.3e-08})\\
\hline
\multicolumn{5}{c}{$20\times20\times20$ grid}  \\
\hline
{\tt 1e-04}  &    114 ({\tt 1.5e-05})  &       99 ({\tt 1.3e-05})  &     75 ({\tt 4.8e-06})   &   174  ({\tt 6.1e-06})\\
{\tt 1e-06}  &    139 ({\tt 1.8e-07})  &      110 ({\tt 8.4e-08})  &     85 ({\tt 1.2e-07})   &   186  ({\tt 6.5e-08})\\
\hline
\multicolumn{5}{c}{$40\times40\times40$ grid}  \\
\hline
{\tt 1e-04}  &    217 ({\tt 3.2e-05})  &      182 ({\tt 2.9e-05})  &    121 ({\tt 2.2e-05})   &   322  ({\tt 3.3e-06})\\
{\tt 1e-06}  &    267 ({\tt 3.1e-07})  &      212 ({\tt 1.5e-07})  &    140 ({\tt 5.9e-08})   &   338  ({\tt 2.2e-08})\\
\hline
\multicolumn{5}{c}{$80\times80\times80$ grid}  \\
\hline
{\tt 1e-04}  &    449 ({\tt 2.7e-05})  &      363 ({\tt 4.8e-05})  &    223 ({\tt 1.9e-05})   &   606  ({\tt 6.5e-06})\\
{\tt 1e-06}  &    521 ({\tt 5.9e-07})  &      410 ({\tt 1.9e-07})  &    249 ({\tt 3.8e-07})   &   626  ({\tt 4.8e-08})\\
\hline
\end{tabular}
\end{table}

Next, we want to gauge the performance of the methods in terms of matrix vector products (matvecs) for varying problem sizes. Precisely, we discretize~\eqref{eq:3d_wave_equation} on grids of size $10 \times 10 \times 10$ to $80 \times 80 \times 80$, leading to matrices of size $1,\!000$ to $512,\!000$. To obtain comparable results, in light of Figure~\ref{fig:tol_vs_acc_isotropic}, we impose the same residual tolerances for all considered methods, namely $10^{-4}$ and $10^{-6}$. The resulting number of matrix vector products and the final relative accuracy of the solution are shown in Table~\ref{tab:3dwave_isotropic} for all combinations of grid size, tolerance and algorithm.

For all choices of grid size and tolerance, the Gautschi cosine scheme is the most efficient of the methods in terms of matrix vector products. For smaller grids, it performs comparably to the RT restarting methods and with increasing grid size it clearly outperforms them, reducing the number of matrix vector products by slightly less than the expected factor of $2$; cf.~the corresponding discussion in Section~\ref{subsec:gautschi}. This reduction in matrix vector products does not come at the cost of a reduced accuracy. In fact, the Gautschi scheme even produces solutions that are slightly more accurate than those produced by the RT schemes. For all considered cases, the two-pass Lanczos method performs worst in terms of matrix vector products, but it delivers a little bit more accurate solutions than the other algorithms for larger grid sizes.

This can be explained by two reasons.  
First, a nonrestarted Lanczos method
for evaluating the matrix functions $\psi$ and $\sigma$ asymptotically 
converges faster than any restarted Lanczos-based method.  Note that
the situation for evaluating the inverse matrix function (i.e.,
for linear system solution) is different: here convergence in restarted 
methods is asymptotically not necessarily worse than with no restarting.
Second, specifically for our RT restarting approach, an error accumulation
may take place, as an additional error (due to a nonzero residual) is
introduced at each restart.

\begin{figure}
\centering
\tikzsetnextfilename{tol_vs_acc_anisotropic}
\pgfplotsset{height=0.43\linewidth,width=0.975\linewidth,compat=1.10,every axis/.append style={legend style={/tikz/every even column/.append style={column sep=6pt}}}}
\pgfplotsset{every tick label/.append style={font=\scriptsize}}

\noindent%
\begin{tikzpicture}[scale=1]%
    \begin{loglogaxis}[legend columns=1,cycle list name=list_std, grid=major, 
   	xlabel={\small residual tolerance}, ylabel={\small relative accuracy}, xmin=5e-9,xmax=2e-1, ymin = 5e-11, ymax = 4e3,
   	title={\small $40 \times 40 \times 40$ anisotropic wave equation},x dir=reverse]

\addplot+[] table [x ={tol},y ={simul}] {figures/data/tol_vs_acc_anisotropic.dat};\addlegendentry{\scriptsize  RT Lanczos, simultaneous}
\addplot+[] table [x ={tol},y ={seq}] {figures/data/tol_vs_acc_anisotropic.dat};\addlegendentry{\scriptsize  RT Lanczos, sequential}
\addplot+[] table [x ={tol},y ={gautschi}] {figures/data/tol_vs_acc_anisotropic.dat};\addlegendentry{\scriptsize  RT Gautschi}
\addplot+[] table [x ={tol},y ={2pl}] {figures/data/tol_vs_acc_anisotropic.dat};\addlegendentry{\scriptsize  Two-pass Lanczos}

\end{loglogaxis}
\end{tikzpicture}
\caption{Relative accuracy of the solution for the anisotropic wave equation computed by the different algorithms, depending on the residual accuracy that is used.}
\label{fig:tol_vs_acc_anisotropic}
\end{figure}
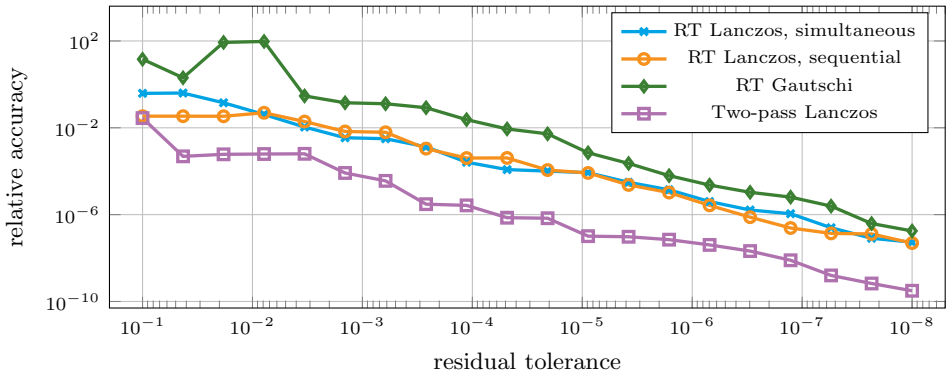

\paragraph{Anisotropic case $k_x = 10^4, k_y = 10^2, k_z = 1$}
Next, we choose the coefficients in the wave equation~\eqref{eq:3d_wave_equation} to be of highly different order of magnitude, $k_x = 10^4, k_y = 10^2, k_z = 1$, leading to a strong anisotropy in the problem. The initial value functions are
$u_0(x,y,z)=\sum_{i,j,k=1}^3 \sin(i\pi x)\sin(j\pi y)\sin(k\pi z)$,
$v_0(x,y,z)=\sum_{i,j,k=1}^3 \lambda_{ijk}$ $\cdot\sin(i\pi x)\sin(j\pi y)\sin(k\pi z)$,
with $\lambda_{ijk}=\pi^2(i^2k_x + j^2k_y + k^2k_z)$.
Apart from that, the setup of this experiment is exactly like in the previous one. We again start by investigating how the residual tolerance influences the final accuracy, with the results depicted in Figure~\ref{fig:tol_vs_acc_anisotropic}. We observe that for this model problem, the accuracy that the different RT schemes deliver is slightly lower than the imposed residual tolerance, with the Gautschi scheme typically being a bit less accurate than the other two RT schemes (and showing signs of instability for low residual tolerances). The two-pass Lanczos algorithm delivers the most accurate solution for all tolerances, which is typically at least an order of magnitude higher than that of the RT schemes.

\begin{table}[]
\small
\centering
\caption{3D anisotropic wave equation test problem. Number of required matvecs and CPU time (and in brackets accuracy reached) by the three schemes for different grids and accuracy requirements. (${}^\star$: The number of matvecs for the two-pass Lanczos method does not fully account for the computational cost of the method, see Remark~\ref{rem:2pl_cost}); ${}^\dagger$: To obtain comparable results, we tighten the tolerance in the RT Gautschi scheme by a factor of 10 and loosen the tolerance in the 2P Lanczos scheme by a factor of 10).} \label{tab:3dwave_anisotropic}
\begin{tabular}{ccccc}
\hline\hline
tolerance & RT simultaneous & RT sequential & RT Gautschi & 2P Lanczos$^\ast$ \\
\hline
\multicolumn{5}{c}{$10\times10\times10$ grid}  \\
\hline
{\tt 1e-04}  &   2355 ({\tt 8.6e-04})  &       60 ({\tt 9.4e-03})  &    899 ({\tt 6.5e-04})   &   201  ({\tt 3.4e-04})\\
 & 0.36s & 0.02s & 0.11s & 0.02s \\
\hdashline
{\tt 1e-06}  &   2854 ({\tt 9.5e-06})  &     1486 ({\tt 3.4e-05})  &   1210 ({\tt 2.7e-06})   &   281  ({\tt 1.4e-06})\\
 & 0.42s & 0.21s & 0.12s & 0.01s \\
\hline
\multicolumn{5}{c}{$20\times20\times20$ grid}  \\
\hline
{\tt 1e-04}  &   4858 ({\tt 5.6e-04})  &     3977 ({\tt 1.0e-03})  &   2424 ({\tt 3.9e-04})   &   441  ({\tt 4.1e-05})\\
 & 1.61s & 1.28s & 0.55s & 0.06s \\
\hdashline
{\tt 1e-06}  &   5322 ({\tt 2.6e-06})  &     4520 ({\tt 7.2e-06})  &   2494 ({\tt 5.7e-06})   &   721  ({\tt 2.1e-06})\\
 & 1.80s & 1.39s & 0.55s & 0.13s \\
\hline
\multicolumn{5}{c}{$40\times40\times40$ grid}  \\
\hline
{\tt 1e-04}  &   9457 ({\tt 2.5e-04})  &     8744 ({\tt 4.7e-04})  &   4794 ({\tt 7.5e-03})   &  1121  ({\tt 3.5e-05})\\
 & 15.04s & 11.36s & 4.65s & 1.10s \\
\hdashline
{\tt 1e-06}  &  10152 ({\tt 8.6e-06})  &     8988 ({\tt 4.0e-06})  &   4993 ({\tt 5.3e-06})   &  2881  ({\tt 3.5e-08})\\
 & 16.62s & 12.14s & 4.41s & 5.43s \\
\hline
\multicolumn{5}{c}{$80\times80\times80$ grid}  \\
\hline
{\tt 1e-04}  &  18745 ({\tt 1.5e-04})  &    19147 ({\tt 3.3e-03})  &  13727 ({\tt 4.5e-04})   &  2281  ({\tt 2.9e-05})\\
 & 189.52s & 146.53s & 109.72 & 14.80s \\
\hdashline
{\tt 1e-06}  &  19421 ({\tt 1.1e-06})  &    17966 ({\tt 4.6e-06})  &   9793 ({\tt 2.8e-06})   &  5241  ({\tt 1.6e-07})\\
 & 193.92s & 130.18s & 66.54s & 48.19s \\
\hline
\end{tabular}
\end{table}

Taking the result of this first investigation into account, we do not impose the same tolerance for all methods in the following when comparing matrix vector products. Instead, we tighten the accuracy requirement by a factor of 10 for the Gautschi scheme (i.e., we use tolerances $10^{-5}$ and $10^{-7}$) and loosen it by a factor of 10 for the two-pass Lanczos method (i.e., we use tolerances $10^{-3}$ and $10^{-5}$). This way, the final accuracy reached by the different solvers is better comparable. Table~\ref{tab:3dwave_anisotropic} shows the number of matrix vector products, required CPU time and the final relative accuracy for all combinations of grid size, tolerance and algorithm. In addition, in Table~\ref{tab:3dwave_anisotropic_t10} we also report the results obtained when approximating the solution of the same problem, but at the later time $t = 10$, which makes the problem more difficult to solve.

As we now report CPU times, a few comments regarding the implementation of the two-pass Lanczos method are in order. When evaluating the stopping criterion, we compute the eigendecomposition of $H_m$ using the MRRR algorithm (via a call to its LAPACK implementation \texttt{dstegr}), and we check the stopping criterion every $10$ iterations. Changing this value affects the performance, and the chosen value seems reasonable to us.

We observe that for $t=1$, the two-pass Lanczos method clearly performs best in terms of matrix vector products and also produces the most accurate solutions. Its run time is lowest among all methods in this experiment, but the difference gets smaller as the problem size increases, indicating that the RT methods show a better scaling behavior. Among the RT schemes, the Gautschi scheme is again the most efficient for larger problem instances, outperforming the other two schemes by a factor of about 2 in terms of matrix vector products. 

% L+M 2023-02-08:
It might seem contra-intuitive in Table~\ref{tab:3dwave_anisotropic} that for the $80 \times 80 \times 80$ grid, the RT Gautschi scheme requires fewer matrix vector products for the stricter tolerance $10^{-6}$ than for 
the looser tolerance $10^{-4}$. This happens because due to the crude tolerance, a too large step size is selected, which can then not be met any longer in later iterations with other starting vectors. In that case, as the Gautschi scheme does not allow to alter the step size in between iterations, the ``step size repair'' described in Section~\ref{subsec:gautschi} needs to be employed for completing the time step. As this step size repair requires the evaluation of both $\psi$ and $\sigma$ functions (and thus building two Krylov subspaces), it typically leads to a larger number of matrix vector products than a Gautschi scheme for a smaller step size which requires no repair step. To avoid this problem, it is in general advisable to use not too crude tolerances in the Gautschi scheme. As the method (as well as the other RT schemes) is rather insensitive to the residual tolerance and scales incredibly well with increasing accuracy requirement (see in particular also the forthcoming Table~\ref{tab:3dwave_anisotropic_t10} which also contains results for tolerance $10^{-8}$), this is typically not expected to lead to high additional costs.

For $t = 10$, the situation differs a bit from that when approximating the solution at $t = 1$. The two-pass Lanczos still needs the smallest number of matrix vector products, but in terms of run time, it is clearly outperformed by the Gautschi scheme (and to a lesser extent by the sequential RT restarting scheme) for larger problem instances. In particular, it is also visible that it shows a worse scaling behavior concerning the CPU time, but also regarding the number of matrix-vector products: In the RT schemes, when increasing the grid resolution by a factor of two, the number of matrix-vector products also increases by a factor of two. In contrast, the two-pass Lanczos method shows an increase by a factor roughly between three and four. However, the final accuracy that two-pass Lanczos reaches is much higher than for the RT schemes, which can likely be attributed to the fact that a larger error accumulation takes place the larger the final time is. We therefore also include results for residual tolerance of $10^{-8}$ in order to illustrate that the RT schemes are indeed capable of reaching higher accuracies and this is not an inherent limitation of the methodology.

\begin{table}[]
\small
\centering
\caption{3D anisotropic wave equation test problem with final time $t=10$. Number of required matvecs and CPU time (and in brackets accuracy reached) by the three schemes for different grids and accuracy requirements. (${}^\star$: The number of matvecs for the two-pass Lanczos method does not fully account for the computational cost of the method, see Remark~\ref{rem:2pl_cost}); ${}^\dagger$: To obtain comparable results, we tighten the tolerance in the RT Gautschi scheme by a factor of 10 and loosen the tolerance in the 2P Lanczos scheme by a factor of 10).} \label{tab:3dwave_anisotropic_t10}
\setlength{\tabcolsep}{5pt}
\begin{tabular}{ccccc}
\hline\hline
tolerance & RT simultaneous & RT sequential & RT Gautschi & 2P Lanczos$^\ast$ \\
\hline
\multicolumn{5}{c}{$10\times10\times10$ grid}  \\
\hline
{\tt 1e-04}  &  24375 ({\tt 8.6e-03})  &    15424 ({\tt 3.0e-02})  &  11081 ({\tt 6.8e-04})   &   361  ({\tt 4.1e-06})\\
 & 2.43s & 1.48s & 1.21s & 0.01s \\
\hdashline
{\tt 1e-06}  &  27576 ({\tt 4.5e-05})  &    19402 ({\tt 1.5e-04})  &  11637 ({\tt 2.0e-05})   &   481  ({\tt 3.4e-08})\\
 & 2.74s & 1.79s & 1.34s & 0.02s \\
\hdashline
{\tt 1e-08}  &  30494 ({\tt 3.2e-07})  &    23824 ({\tt 3.0e-07})  &  14415 ({\tt 2.4e-08})   &   561  ({\tt 2.2e-08})\\
 & 3.04s & 2.23s & 1.39s & 0.03s \\
\hline
\multicolumn{5}{c}{$20\times20\times20$ grid}  \\
\hline
{\tt 1e-04}  &  50279 ({\tt 5.5e-03})  &    43333 ({\tt 7.0e-03})  &  23645 ({\tt 1.0e-03})   &  1321  ({\tt 4.8e-06})\\
 & 11.79s & 9.52s & 5.70s & 0.30s \\
\hdashline
{\tt 1e-06}  &  56040 ({\tt 4.0e-05})  &    46718 ({\tt 6.9e-05})  &  24429 ({\tt 1.2e-06})   &  1841  ({\tt 7.3e-08})\\
 & 13.71s & 10.25s & 5.79s & 0.71s \\
\hdashline
{\tt 1e-08}  &  56724 ({\tt 2.2e-07})  &    48515 ({\tt 8.2e-08})  &  28783 ({\tt 3.5e-09})   &  2361  ({\tt 1.8e-09})\\
 & 13.88s & 10.30s & 6.53s & 1.36s \\
\hline
\multicolumn{5}{c}{$40\times40\times40$ grid}  \\
\hline
{\tt 1e-04}  &  95485 ({\tt 1.8e-03})  &    89023 ({\tt 4.5e-03})  &  47917 ({\tt 1.9e-04})   &  6041  ({\tt 4.5e-06})\\
 & 101.26s & 81.19s & 42.82s & 23.59s \\
\hdashline
{\tt 1e-06}  &  105654 ({\tt 7.8e-05})  &    93436 ({\tt 6.3e-05})  &  49480 ({\tt 9.7e-07})   &  9641  ({\tt 8.1e-08})\\
 & 105.08s & 80.43s & 45.74s & 89.27s \\
\hdashline
{\tt 1e-08}  &  108415 ({\tt 5.1e-07})  &    96162 ({\tt 1.8e-07})  &  49534 ({\tt 2.0e-09})   & 16841  ({\tt 1.2e-09})\\
 & 105.83s & 82.59s & 46.93s & 504.73s \\
\hline
\multicolumn{5}{c}{$80\times80\times80$ grid}  \\
\hline
{\tt 1e-04}  &  191517 ({\tt 1.6e-03})  &    179896 ({\tt 2.6e-03})  &  95788 ({\tt 1.5e-03})   & 14441  ({\tt 1.1e-06})\\
 & 1764.20s & 1095.76s & 647.46s & 370.60s \\
\hdashline
{\tt 1e-06}  &  199654 ({\tt 2.2e-05})  &    177549 ({\tt 4.4e-05})  &  98384 ({\tt 6.5e-05})   & 25801  ({\tt 1.3e-08})\\
 & 1859.11s & 1086.36s & 616.70s & 1898.00s \\
\hdashline
{\tt 1e-08}  &  210882 ({\tt 7.0e-07})  &    192096 ({\tt 2.6e-07})  &  99024 ({\tt 8.0e-08})   & 41201  ({\tt 2.3e-09})\\
 & 1986.96s & 1193.28s & 616.48s & 8275.04s \\
\hline
\end{tabular}
\end{table}

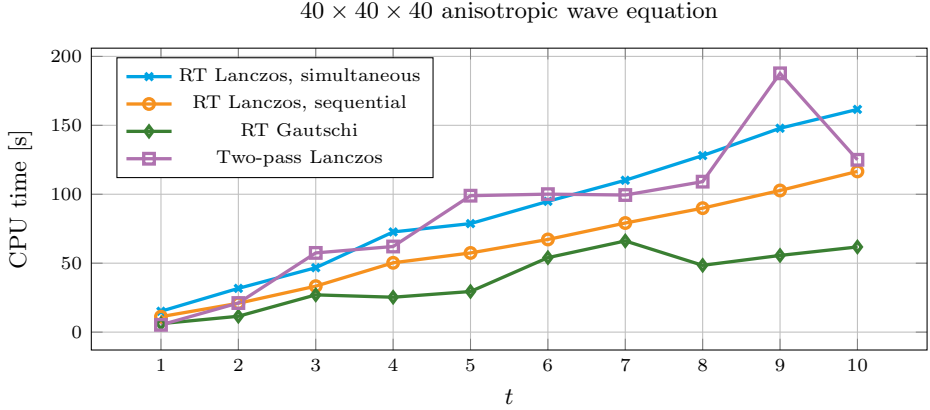
\begin{figure}
\centering
\tikzsetnextfilename{t_vs_time_anisotropic}
\pgfplotsset{height=0.43\linewidth,width=0.975\linewidth,compat=1.10,every axis/.append style={legend style={/tikz/every even column/.append style={column sep=6pt}}}}
\pgfplotsset{every tick label/.append style={font=\scriptsize}}

\noindent%
\begin{tikzpicture}[scale=1]%
    \begin{axis}[legend columns=1,cycle list name=list_std, grid=major, 
   	xlabel={\small $t$}, ylabel={\small CPU time [s]}, legend pos = north west, %xmin=5e-9,xmax=2e-1, ymin = 5e-11, ymax = 4e3,
   	title={\small $40 \times 40 \times 40$ anisotropic wave equation}]

\addplot+[] table [x ={t},y ={simul}] {figures/data/t_vs_time_anisotropic.dat};\addlegendentry{\scriptsize  RT Lanczos, simultaneous}
\addplot+[] table [x ={t},y ={seq}] {figures/data/t_vs_time_anisotropic.dat};\addlegendentry{\scriptsize  RT Lanczos, sequential}
\addplot+[] table [x ={t},y ={gautschi}] {figures/data/t_vs_time_anisotropic.dat};\addlegendentry{\scriptsize  RT Gautschi}
\addplot+[] table [x ={t},y ={2pl}] {figures/data/t_vs_time_anisotropic.dat};\addlegendentry{\scriptsize  Two-pass Lanczos}

\end{axis}
\end{tikzpicture}
\caption{CPU time used by the different algorithms for solving the anisotropic wave equation, depending on the final time $t$.}
\label{fig:t_vs_time_anisotropic}
\end{figure}

In Figure~\ref{fig:t_vs_time_anisotropic}, we illustrate how the CPU times evolve for the $40 \times 40 \times 40$ grid for final times between the two values $t=1$ to $t=10$ covered in Tables~\ref{tab:3dwave_anisotropic} and~\ref{tab:3dwave_anisotropic_t10}, when imposing a residual tolerance of $10^{-6}$. One can observe that while for $t=1$, the run time of two-pass Lanczos and Gautschi methods are quite close to each other (as is already visible from Table~\ref{tab:3dwave_anisotropic}), for higher values of $t$ (i..e, for more difficult problems), the Gautschi scheme always outperforms the other schemes including the two-pass method.

\begin{figure}
\centering
\tikzsetnextfilename{tol_vs_acc_transport}
\pgfplotsset{height=0.43\linewidth,width=0.975\linewidth,compat=1.10,every axis/.append style={legend style={/tikz/every even column/.append style={column sep=6pt}}}}
\pgfplotsset{every tick label/.append style={font=\scriptsize}}

\noindent%
\begin{tikzpicture}[scale=1]%
    \begin{loglogaxis}[legend columns=1,cycle list name=list_std, grid=major, 
   	xlabel={\small residual tolerance}, ylabel={\small relative accuracy},  xmin=5e-9,xmax=1e-1,
   	title={\small Transport equation with decay, grid size $512$},x dir=reverse]

\addplot+[] table [x ={tol},y ={simul}] {figures/data/tol_vs_acc_transport.dat};\addlegendentry{\scriptsize  RT Arnoldi, simultaneous}
\addplot+[] table [x ={tol},y ={seq}] {figures/data/tol_vs_acc_transport.dat};\addlegendentry{\scriptsize  RT Arnoldi, sequential}
\addplot+[] table [x ={tol},y ={gautschi}] {figures/data/tol_vs_acc_transport.dat};\addlegendentry{\scriptsize  RT Gautschi}
\addplot+[] table [x ={tol},y ={firstorder}] {figures/data/tol_vs_acc_transport.dat};\addlegendentry{\scriptsize  first order RT}

\end{loglogaxis}
\end{tikzpicture}
\caption{Relative accuracy of the solution for the transport equation with decay computed by the different algorithms, depending on the residual accuracy that is used.}
\label{fig:tol_vs_acc_transport}
\end{figure}
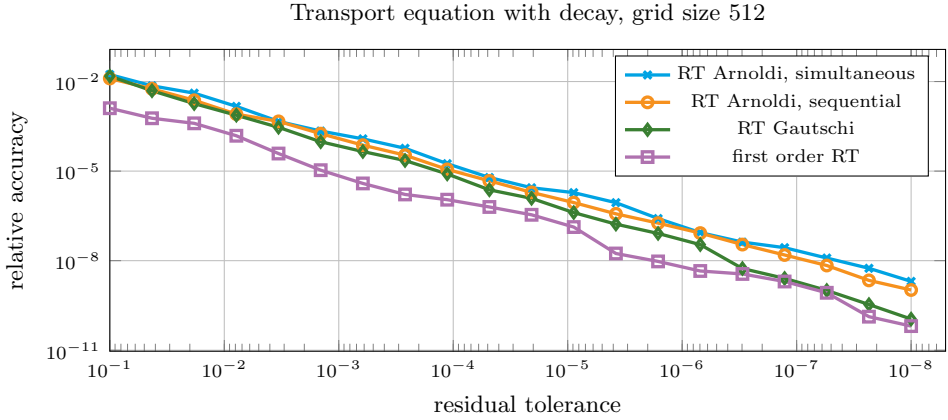

\subsection{Transport equation with decay}
We now turn to a problem resulting in a nonsymmetric matrix in~\eqref{ivp}. Consider the following transport equation with decay,
\begin{equation}\label{eq:transport_decay}
\left\{\begin{aligned}
& u_t = -c u_x - \alpha u, \\
& u(0,x) = u_0(x)
\end{aligned}\right.
\end{equation}
with parameters $c, \alpha > 0$ and a given function $u_0: \mathbb{R} \longrightarrow \mathbb{R}$. Assuming sufficient smoothness in $u$ and $u_0$, equation~\eqref{eq:transport_decay} can be turned into a second-order PDE
\begin{equation}\label{eq:transport_decay_second_order}
\left\{\begin{aligned}
& u_{tt} = c^2 u_{xx} + 2c\alpha u_x + \alpha^2 u, \\
& u(0,x) = u_0(x), \quad u_t(0,x) = u_0^\prime(x) - \alpha u_0(x).
\end{aligned}\right.
\end{equation}
We want to solve~\eqref{eq:transport_decay_second_order} on the domain $\Omega = (0,1)$, imposing homogeneous Dirichlet boundary conditions. Discretizing both the first and  second spatial derivative by second-order centralized finite differences with $n_x$ interior grid points leads to an IVP of the form~\eqref{ivp} with $A = -c^2L_x - 2\alpha cD_x - \alpha^2 I \in \mathbb{R}^{n_x \times n_x}$, where $L_x$ is defined in~\eqref{eq:Laplacian_discretization} and
$$D_x = \frac{1}{2h_x}\text{tridiag}(-1,0,1) \in \mathbb{R}^{n_x \times n_x}.$$
Clearly, as $L$ and $I$ are symmetric and $D$ is nonsymmetric, the resulting matrix $A$ is nonsymmetric. Note that we need to choose $\alpha$ small enough so that $-c^2L_x - \alpha^2 I$, the symmetric part of $A$, is positive semidefinite.

In our experiment we choose $c = 0.3, \alpha = 1$ and the initial conditions $u_0(x) = e^{-500(x-0.5)^2}$ and want to approximate the solution $u$ at time $t=1$. We compare both versions of the RT restarting Arnoldi method and the Gautschi scheme and impose the same accuracy requirements as before. As the problem is nonsymmetric, we cannot use a two-pass Lanczos approach here. Instead, to have another method to compare to, we use the residual time restarting method for the matrix exponential from~\cite{BotchevKnizhnerman2020}\footnote{available at \url{https://team.kiam.ru/botchev/expm/}}, applied to the equivalent first order ODE system~\eqref{ivp1}. As the matrix $\mathcal{A}$ defined in~\eqref{mtx2} is of size $2n \times 2n$, this requires working with vectors of twice the size, as already explained in Section~\ref{subsec:firstoder}. Thus, to have orthogonalization cost and memory footprint comparable to our methods that work directly on the second order formulation, we run the method with maximum Krylov dimension $m/2$. Note that while technically, matrix vector products with a matrix of twice the size have to be computed, the simple structure of $\mathcal{A}$ allows them to be performed at a cost that is only slightly higher than that for multiplication with $A$. We therefore do not distinguish between matrix vector products with $A$ and $\mathcal{A}$ when reporting our results.

As for the other test cases, we start by investigating how the residual tolerance translates into relative accuracy in the solution, see Figure~\ref{fig:tol_vs_acc_transport}, where we have used $512$ discretization points. We can observe that for this model problem, the three schemes working on the second-order formulation produce solutions with relative accuracy in the same order of magnitude as the imposed residual tolerance (or even slightly more accurate). Interestingly, the Gautschi cosine scheme delivers the most accurate solution among those methods, while the simultaneous RT Arnoldi scheme is slightly less accurate than the sequential version. The first order RT scheme working on the formulation~\eqref{ivp1} is about one order of magnitude more accurate than the other methods.

\begin{table}[]
\small
\centering
\caption{1D transport equation with decay test problem. Number of matvecs required (and in brackets accuracy reached) by the four schemes for different grids and accuracy requirements (${}^\dagger$: To obtain comparable results, we loosen the tolerance in the first order RT scheme by a factor of 10).}\label{tab:1dtransport}\begin{tabular}{ccccc}
\hline\hline
tolerance${}^\dagger$ & RT simultaneous & RT sequential & RT Gautschi & first order RT \\
\hline
\multicolumn{5}{c}{ grid size $128$}  \\
% L+M 2023.01.19:
\hline
{\tt 1e-04}  &     91 ({\tt 3.5e-05})  &       86 ({\tt 8.9e-06})  &     69 ({\tt 3.4e-06}) &    104 ({\tt 3.4e-05})  \\
{\tt 1e-06}  &    110 ({\tt 3.3e-07})  &       96 ({\tt 1.2e-07})  &     74 ({\tt 2.3e-08}) &    134 ({\tt 2.8e-07})  \\
\hline
\multicolumn{5}{c}{ grid size $256$}  \\
% L+M 2023.01.19:
\hline
{\tt 1e-04}  &    168 ({\tt 1.8e-05})  &      154 ({\tt 1.0e-05})  &    103 ({\tt 2.6e-06}) &    191 ({\tt 1.6e-05})  \\
{\tt 1e-06}  &    191 ({\tt 2.2e-07})  &      169 ({\tt 1.1e-07})  &    111 ({\tt 2.0e-08}) &    239 ({\tt 2.7e-07})  \\
\hline
\multicolumn{5}{c}{ grid size $512$}  \\
% L+M 2023.01.19:
\hline
{\tt 1e-04}  &    317 ({\tt 1.5e-05})  &      293 ({\tt 1.1e-05})  &    221 ({\tt 6.0e-06}) &  374 ({\tt 5.0e-06})  \\
{\tt 1e-06}  &    350 ({\tt 1.5e-07})  &      319 ({\tt 1.0e-07})  &    223 ({\tt 6.1e-08}) &  465 ({\tt 1.6e-07})  \\
\hline
\multicolumn{5}{c}{ grid size $1024$}  \\
% L+M 2023.01.19:
\hline
{\tt 1e-04}  &    635 ({\tt 1.7e-05})  &      582 ({\tt 1.7e-05})  &    451 ({\tt 6.8e-06}) &  885 ({\tt 4.6e-06})  \\
{\tt 1e-06}  &    674 ({\tt 1.5e-07})  &      619 ({\tt 9.5e-08})  &    436 ({\tt 4.2e-08}) &  829 ({\tt 3.8e-08})  \\
\hline
\end{tabular}
\end{table}

For comparing the required number of matrix vector products for the different algorithms, we consider discretizations of~\eqref{eq:transport_decay_second_order} on grids of size 128 to 1024. Let us note that the problem sizes occurring in this example of course do not necessitate using restarts in practice, but it is still instructive to study the behavior of the different algorithms for a nonsymmetric, ill conditioned matrix ($\kappa(A) \approx 1.7 \cdot 10^5$). We report the number of matrix vector products and the final relative accuracy of the solution in Table~\ref{tab:1dtransport}, again for all combinations of grid size, tolerance and algorithm. In light of the observations reported in Figure~\ref{fig:tol_vs_acc_transport}, we loosen the tolerance in the first order RT scheme by a factor of 10 to obtain better comparable results.

We observe that the Gautschi scheme outperforms the other RT methods, as it requires fewer matrix vector products (typically saving about a factor 1.4) while at the same time yielding a higher accuracy. Between the two RT restarting schemes, the sequential version shows slightly better performance both in terms of matrix vector products and final accuracy. The first order RT scheme requires the highest number of matrix vector products in all test cases. For the smaller grid sizes, it delivers a solution which is about one order of magnitude less accurate than that of the Gautschi scheme, while for the largest grid it is slightly more accurate (but also requires almost twice the number of matrix vector products).

\section{Conclusions}\label{sec:conclusions}
Several conclusions can be made.  
\begin{enumerate}
\item
The introduced residual concept appears to be a reliable way to control
convergence in Krylov subspace methods for evaluating the $\psi$ and
$\sigma$ matrix functions.
\item
Qualitative behavior of the residuals as a function of time allows to
extend the residual-time (RT) Krylov subspace restarting technique to the $\psi$ and
$\sigma$ matrix functions. The proposed residual-time (RT) restarting
proves to be robust and efficient.
\item 
An important interesting conclusion is that an RT restarted Lanczos process 
for evaluating  the $\psi$ and $\sigma$ matrix functions
can be more efficient than the non-restarted two-pass Lanczos process.
The efficiency gain is not only due to the increased
overhead costs in the non-restarted Lanczos process to evaluate the projected
tridiagonal matrix but also in terms of total number of 
required matrix vector products. 
\item
Compared to solving second-order IVPs by straightforward evaluation of the 
$\psi$ and $\sigma$ matrix functions, the Gautschi time integration scheme 
proves to be more efficient.  As our experiments indicate, when both
integration methods are implemented with incorporated residual-based
stopping criterion and RT restarting, the Gautschi scheme typically requires
up to a factor two fewer matrix vector multiplications.
\end{enumerate}

One important research question which can be addressed is a proper choice 
of the restart length.  Another effect which also should be studied is an 
error accumulation in the Gautschi scheme due to the inexact Krylov subspace
$\psi$ function evaluations.  We hope to be able to address these questions in
the future.

\paragraph{Acknowledgments} We want to thank the anonymous referees for their careful reading and for making several suggestions that improved the manuscript.

\appendix
\section{Derivation of~\eqref{yex}}\label{sec:appendix_new}
In this section, we show that~\eqref{yex} solves the initial value problem~\eqref{ivp}.

First, observe that the initial condition on $\bm{y}(0)$ is fulfilled, as
\[\bm{y}(0) = \bm{u} + 0\cdot\psi(0\cdot A)(-A\bm{u}+\bm{g}) + 0\cdot \sigma(0 \cdot A)\bm{v} = \bm{u}.\]
For checking the other relations, we need the first two derivatives of $t^2\psi(t^2A)$ and $t\sigma(t^2A)$. 
The first derivatives are given in~\eqref{eq:first_derivatives_psi_sigma}.  Hence,
\begin{equation}\label{eq:second_derivatives_psi_sigma}
\begin{aligned}
\frac{\text{d}^2}{\text{d} t^2}\ \frac12t^2\psi(t^2A) &= \frac{\text{d}}{\text{d} t}\  t\sigma(t^2A) = I-\frac12t^2A\psi(t^2 A), \\
\frac{\text{d}^2}{\text{d} t^2}\  t\sigma(t^2A) &= \frac{\text{d}}{\text{d} t}\ \left(I-\frac12t^2A\psi(t^2 A)\right) = -tA\sigma(t^2A).
\end{aligned}
\end{equation}
Using~\eqref{eq:first_derivatives_psi_sigma}, we have
\begin{align*}
\bm{y}^\prime(t) &= 
\frac{\text{d}}{\text{d} t}\left( \frac12 t^2\psi(t^2 A)(-A\bm{u}+\bm{g}) + t\sigma(t^2A)\bm{v}\right) \\
&= t\sigma(t^2A)(-A\bm{u}+\bm{g}) + \left(I-\frac12t^2A\psi(t^2 A)\right) \bm{v}\\
&= \bm{v} + t\sigma(t^2A)(-A\bm{u}+\bm{g}) -\frac12t^2A\psi(t^2 A) \bm{v},
\end{align*}
from which $\bm{y}^\prime(0) = \bm{v}$ follows. Finally, using~\eqref{eq:second_derivatives_psi_sigma}, we find
\begin{align*}
\bm{y}^{\prime\prime}(t) &= \frac{\text{d}^2}{\text{d} t^2}\left( \frac12 t^2\psi(t^2 A)(-A\bm{u}+\bm{g}) + t\sigma(t^2A)\bm{v} \right)\\
&= \left(I-\frac12t^2A\psi(t^2 A)\right)(-A\bm{u}+\bm{g}) - tA\sigma(t^2A)\bm{v} \\
&= -A\bm{u} -\frac12t^2A\psi(t^2 A)(-A\bm{u} +\bm{g}) - tA\sigma(t^2A)\bm{v} + \bm{g}\\
&= -A\left(
(\bm{u} + \frac12t^2\psi(t^2 A)(-A\bm{u} +\bm{g}) + t\sigma(t^2A)\bm{v})
\right)+ \bm{g}\\
&= -A\bm{y}(t) + \bm{g},
\end{align*}
where, at the last step, relation~\eqref{yex} is used.

\section{Proofs of Propositions~\ref{P2}, \ref{P3}, and~\ref{P4}}\label{sec:appendix}
\begin{proof}[Proof of Proposition~\ref{P2}]
Since $\frac{\partial^2 f}{\partial t^2} = \cos(t\sqrt{z})$,
function (\ref{fun2}) satisfies the differential equation in $t$
\be \label{diffeq2}
\frac{\partial^2 f}{\partial t^2} + zf - 1 = 0.
\ee
Compared to equation~(\ref{diffeq1}), equation~(\ref{diffeq2}) contains an extra
constant term. This forces us to specially handle zero degree terms in the 
remaining part of the proof.

As $\Phi_0(z)=1$, it follows from (\ref{diffeq2}) that
\be \label{zeroser}
0 = [f_0(t)z+f_0^{\prime\prime}(t)-1]\Phi_0(z)
+ \sum_{j=1}^\infty [f_j(t)z+f_j^{\prime\prime}(t)]\Phi_j(z).
\ee
All the Faber coefficients in the right-hand side of (\ref{zeroser}) must be zero.
We then derive for the associated residual
\beas
& &-\bm{r}_m(t) \\
&=& \bm{y}_m^{\prime\prime}(t)+A\bm{y}_m(t)-\bm{v}_1^{(\psi)} \\
&=& V_m^{(\psi)}\sum_{j=0}^\infty f_j^{\prime\prime}(t)\Phi_j(H_m)\bm{e}_1 + AV_m^{(\psi)}\sum_{j=0}^\infty f_j(t)\Phi_j(H_m)\bm{e}_1
-V_m^{(\psi)}\Phi_0(H_m)\bm{e}_1 \\
&=&(V_m^{(\psi)}H_m+h_{m+1,m}\bm{v}_{m+1}^{(\psi)}\bm{e}_m^\tT) \sum_{j=0}^\infty f_j(t)\Phi_j(H_m)\bm{e}_1\\
& &+ V_m^{(\psi)}\sum_{j=0}^\infty f_j^{\prime\prime}(t)\Phi_j(H_m)\bm{e}_1 - V_m^{(\psi)}\Phi_0(H_m)\bm{e}_1\\
&=& V_m^{(\psi)}\sum_{j=1}^\infty [f_j^{\prime\prime}(t){I}_m+H_mf_j(t)]\Phi_j(H_m)\bm{e}_1 + V_m^{(\psi)}[f_0^{\prime\prime}(t){I}_m+H_mf_0(t)-{I}_m]\Phi_0(H_m)\bm{e}_1 \\
& & + h_{m+1,m}\bm{v}_{m+1}^{(\psi)}\bm{e}_m^\tT \sum_{j=0}^\infty f_j(t)\Phi_j(H_m)\bm{e}_1\\
&=&h_{m+1,m}\bm{v}_{m+1}^{(\psi)}\bm{e}_m^\tT \sum_{j=m-1}^\infty f_j(t)\Phi_j(H_m)\bm{e}_1.
\eeas
This again implies~(\ref{appr2}).
\end{proof}

\begin{proof}[Proof of Proposition~\ref{P3}]
Thanks to (\ref{appr2}), (\ref{FabCheb}) and (\ref{resL1}), we have
\beas
\|\bm{r}_m(t)\| &\le& 2\sum_{j=m-1}^\infty|f_j(t)| 
\le 8\sum_{j=m-1}^\infty \left| \sum_{l=j}^\infty J_{2l+1}(t) \right|
\le 8\sum_{j=m-1}^\infty \sum_{l=j}^\infty \left| J_{2l+1}(t) \right| \\
&=& 8 \sum_{l=m-1}^\infty (l-m+2)|J_{2l+1}(t)|.
\eeas
Applying \cite[formula~9.1.62]{AbramowitzStegun1964} gives
\[
\sum_{l=m-1} (l-m+2)|J_{2l+1}(t)|
\le \sum_{l=m-1}^\infty (l-m+2)\frac{(t/2)^{2l+1}}{(2l+1)!}.
\]
Let us majorize the last series with a geometric one. Since
\beas
\frac{l-m+3}{l-m+2} \cdot \frac{(t/2)^{2l+3}}{(t/2)^{2l+1}} \cdot
\frac{(2l+1)!}{(2l+3)!}
= \frac{l-m+3}{l-m+2} \cdot \left(\frac{t}{2}\right)^2 \cdot
\frac{1}{(2l+2)(2l+3)} \\
< \frac{t^2}{2(2l+2)(2l+3)} < \frac{1}{2},
\eeas
we obtain
\[
\sum_{l=m-1}^\infty (l-m+2)\frac{(t/2)^{2l+1}}{(2l+1)!}
\le 2\frac{(t/2)^{2m-1}}{(2m-1)!}.
\]
Accounting this yields (\ref{resP3}).
\end{proof}

\begin{proof}[Proof of Proposition~\ref{P4}]
In view of (\ref{resL2}) we have
\beas
\|\bm{r}_m(t)\| &\le& 8\sum_{j=m-1}^\infty \sum_{l=0}^\infty (l+1)|J_{2(j+l+1)}(t)|
= 8\sum_{k=m}^\infty \left[\sum_{l=0}^{k-m}(l+1)\right]|J_{2k}(t)| \\
&=&8 \sum_{k=m}^\infty \frac{(k-m+2)(k-m+1)}{2} |J_{2k}(t)|\\
&=&4 \sum_{k=m}^\infty (k-m+2)(k-m+1) |J_{2k}(t)| \\
&\le& 4 \sum_{k=m}^\infty (k-m+2)(k-m+1) \frac{(t/2)^{2k}}{(2k)!}.
\eeas
Since
\beas
& &(k+1-m+2)(k+1-m+1)\frac{(t/2)^{2k+2}}{(2k+2)!}
:\left[(k-m+2)(k-m+1)\frac{(t/2)^{2k}}{(2k)!}\right] \\
&=& \frac{k-m+3}{k-m+1}\cdot\left(\frac{t}{2}\right)^2\cdot\frac{1}{(2k+1)(2k+2)}
\le \frac{3t^2}{4\cdot3\cdot4} \le \frac{1}{16},
\eeas
we obtain (\ref{resP4}):
\[
\|\bm{r}_m(t)\| \le 4\cdot\frac{16}{15}\cdot2\cdot\frac{(t/2)^{2m}}{(2m)!}
=\frac{128}{15}\cdot\frac{(t/2)^{2m}}{(2m)!}.
\]
\end{proof}

\bibliography{matrixfunctions,Mike_bib}
\bibliographystyle{abbrv}

\end{document}